\documentclass[12pt,a4paper]{article}

\usepackage{pgf,pgfarrows,pgfnodes,pgfautomata,pgfheaps,pgfshade}
\usepackage{amsmath,amssymb,amsfonts,amsthm,comment}
\usepackage[latin1]{inputenc}
\usepackage{colortbl}
\usepackage[english]{babel}

\usepackage{lmodern}
\usepackage[T1]{fontenc} 

\usepackage{times}
\usepackage{tikz}
\usepackage{bm}
\usepackage{color}
\usepackage{tabu}
\usepackage{graphicx}
\usepackage{float}
\usepackage[latin1]{inputenc}
\usepackage{setspace}
\usepackage{afterpage}

\def\Z{\ifmmode{\mathbb Z}\else{$\mathbb Z$}\fi}
\def\C{\ifmmode{\mathbb C}\else{$\mathbb C$}\fi}
\def\Q{\ifmmode{\mathbb Q}\else{$\mathbb Q$}\fi}
\def\K{\ifmmode{\mathbb K}\else{$\mathbb K$}\fi}
\def\P{\ifmmode{\mathbb P}\else{$\mathbb P$}\fi}
\def\R{\ifmmode{\mathbb R}\else{$\mathbb R$}\fi}
\def\H{\ifmmode{\mathbb H}\else{$\mathbb H$}\fi}
\def\g{\ifmmode{\mathfrak g}\else {$\mathfrak g$}\fi}
\def\h{\ifmmode{\mathfrak h}\else {$\mathfrak h$}\fi}
\def\a{\ifmmode{\mathfrak a}\else {$\mathfrak a$}\fi}
\def\k{\ifmmode{\mathfrak k}\else {$\mathfrak k$}\fi}
\def\p{\ifmmode{\mathfrak p}\else {$\mathfrak p$}\fi}
\def\b{\ifmmode{\mathfrak b}\else {$\mathfrak b$}\fi}
\def\n{\ifmmode{\mathfrak n}\else {$\mathfrak n$}\fi}
\def\m{\ifmmode{\mathfrak m}\else {$\mathfrak m$}\fi}
\def\t{\ifmmode{\mathfrak t}\else {$\mathfrak t$}\fi}
\def\O{\ifmmode{\mathcal{O}}\else {$\mathcal{O}$}\fi}
\def\W{\ifmmode{\mathcal{W}}\else {$\mathcal{W}$}\fi}
\def\so{\ifmmode{\mathfrak {so}(n)} \else {$\mathfrak {so} (n)$}\fi}
\def\soc{\ifmmode{\mathfrak {so}(2l,\C)} \else {$\mathfrak {so} (n,\C)$}\fi}
\def\socc{\ifmmode{\mathfrak {so}(2l+1,\C)} \else {$\mathfrak {so} (n,\C)$}\fi}
\def\u {\ifmmode{\mathfrak {u}(n)} \else {$\mathfrak {u} (n)$}\fi}
\def\su {\ifmmode{\mathfrak {su}(n)} \else {$\mathfrak {su} (n)$}\fi}
\def\sp {\ifmmode{\mathfrak {sp}(n)} \else {$\mathfrak {sp} (n)$}\fi}
\def\spc {\ifmmode{\mathfrak {sp}(2l,\C)} \else {$\mathfrak {sp} (n,\C)$}\fi}
\def\slr {\ifmmode{\mathfrak {sl}(l,\R)} \else {$\mathfrak {sl} (n,\R)$}\fi}
\def\sl {\ifmmode{\mathfrak {sl}(l,\C)} \else {$\mathfrak {sl} (n,\C)$}\fi}
\def\sll {\ifmmode{\mathfrak {sl}(2l,\C)} \else {$\mathfrak {sl} (n,\C)$}\fi}
\def\slll {\ifmmode{\mathfrak {sl}(2l+1,\C)} \else {$\mathfrak {sl} (n,\C)$}\fi}
\def\slh {\ifmmode{\mathfrak {sl}(m,\H)} \else {$\mathfrak {sl} (n,\H)$}\fi}
\def\sup {\ifmmode{\mathfrak {su}(p,q)} \else {$\mathfrak {su} (p,q)$}\fi}
\def\gl {\ifmmode{\mathfrak {gl}(l,\R)} \else {$\mathfrak {gl} (n,\R)$}\fi}
\def\glc{\ifmmode{\mathfrak {gl}(l,\C)} \else {$\mathfrak {gl} (n,\C)$}\fi}
\def\glh{\ifmmode{\mathfrak {gl}(l,\H)} \else {$\mathfrak {gl} (n,\H)$}\fi}

\def\diag{{\rm diag}}

\def\SLC{{SL(l,\C)}}

\DeclareMathOperator{\tr}{tr}

\newtheorem{thm}{Theorem}
\newtheorem{prop}[thm]{Proposition}
\newtheorem{defi}[thm]{Definition}
\newtheorem{lemma}[thm]{Lemma}
\newtheorem{cor}[thm]{Corollary}

\begin{document}


\title{Cycle transversal flag domains of Hodge type \footnote{Supported by SFB/TR 45 of the Deutsche Forschungsgemeinschaft}  }
\author{Ana-Maria Brecan}
\date{}
\maketitle

\section{Introduction}
Let $V$ be a finite dimensional $\Q$- or $\R$-vector space and $D$ a classifying space for a variation of Hodge structure on $V$, or equivalently a period domain for a given weight $n$, a given set of Hodge numbers and a reference Hodge decomposition on $V_\C$, the complexification of $V$. It is well known that $D$ can be identified with a flag domain, i.e. an open $G_0$-orbit inside a flag manifold $Z:=G/P$, where $G_0$ is a certain real simple Lie group viewed as a real form of a complex simple Lie group $G$ of  classical type (B), (C) or (D) and $P$ a parabolic subgroup. Choose a base point $z_0$ in $D$, i.e. a Hodge filtration corresponding to a Hodge decomposition $V_\C:=\oplus_{p+q=n}V^{p,q}$ such that $V^{q,p}=\overline{V}^{p,q}$, with $P$  the isotropy subgroup of $G$ at $z_0$. Then one can show that $P_0:=P\cap G_0$, the isotropy subgroup of $G_0$ at $z_0$, is a compact subgroup of $G_0$ containing a compact maximal torus $T_0$ of $G_0$. Let $K_0$ denote the unique maximal compact subgroup of $G_0$ containing $P_0$, $\k_0$ its Lie algebra and $\theta$ the Cartan involution on $\g_0$ which is equal to the identity on $\k_0$ and $-$ the identity on $\mathfrak{m}_0$, the orthogonal complement of $k_0$ in $\g_0$ (with respect to the Killing form).  Then $C_0:=K_0/P_0$ is the unique $K_0$-orbit in $D$ which is a closed complex submanifold of $D$. In the language of flag domains and their associated cycle spaces this is called a (base) cycle, see \cite{Fels2006}.
\newline
\noindent
\newline
The base Hodge decomposition on $V_\C$ induces a Hodge decomposition of weight 0 on the Lie algebra $\g$ of $G$ defined by $\g:=\oplus_{f=-n}^{f=n}\g^{-f,f}, $ where for fixed $f$, $\g^{-f,f}$ is defined as the set of endomorphisms $g\in \g$ of $V_{\C}$ which act on the Hodge decomposition by sending $V^{p,q}$ to $V^{p-f,q+f}$, for all $p,q$ with $p+q=n$. As such the Lie algebra $\p$ of $P$ decomposes as $\p=\oplus_{f=-n}^0 \g^{-f,f}$, with Levi component $\mathfrak{s}:=\g^{0,0}$, unipotent radical $\mathfrak{u}^+:=\oplus _{f=-n}^{-1}\g^{-f,f}$ and opposite unipotent radical $\mathfrak{u}^{-}:=\oplus_{f=1}^n \g^{-f,f}$. The holomorphic tangent space of $D$ at $z_0$ can be naturally identified with $\g/\p$ which in turn can be naturally identified with $\mathfrak{u}^-$. Furthermore, if we denote by $\k$ and $\mathfrak{m}$ respectively, the complexification of $\k_0$ and $\mathfrak{m}_0$ respectively, and simply by $\theta$ the complex linear extension of $\theta$ to $\g$, then one can write $\g$ as a direct sum $\g=\k\oplus\mathfrak{m}$, where $\k=\oplus_{f \equiv 0 (\text{mod } 2)}\g^{-f,f}$ and $\mathfrak{m}=\oplus_{f\equiv 1(\text{mod }2)}\g^{-f,f}$. Furthermore, this decomposition is real and if $K$ is the associated Lie subgroup of $G$ with Lie algebra $\k$, then  $K.z$ is the unique minimal dimensional, closed orbit of $K$ in $Z$, itself a flag manifold, which is Matzuki-dual to $D$. Furthermore, this orbit is precisely the base cycle. 
\newline
\noindent
\newline
Now let $G_0'$ be a Lie group of Hodge type, i.e. a real Lie group which contains a compact maximal torus. Furthermore assume that $G_0'$ is semisimple and let $D'=G_0'/P_0'$ be a $G_0'$-flag domain in $Z'=G'/P'$, with $G'$ the complexification of $G_0'$ and $P'$ the isotropy subgroup of $G'$ at $z_0'\in D'$ such that $P_0'=P'\cap G_0'$ is compact. The holomorphic tangent space of $D'$ at $z_0'$ can be naturally identified with $\g'/\p'$ which in turn can be naturally identified with $\mathfrak{u'}^-$, the opposite unipotent radical of $\p'$. The problem that we shall consider in this paper \footnote{Throughout this paper refer to this problem as the classification problem (*)} is that of classifying all proper, equivariant embeddings $f:D'\rightarrow D$ with $D$ an arbitrary period domain and $D'$ a flag domain of a group of Hodge type such that the following assumptions are satisfied:
\begin{enumerate}
\item{The map $f$ induces an embedding of $Z'$ as a compact submanifold of $Z$ such that $G_0'$ and $G'$ respectively are identified with closed subgroups of $G_0$ and $G$ respectively and $P_0\subset P_0'$, $P\subset P'$, i.e. $f(z_0')=z_0$.}
\item{The induced map at the Lie algebra level $\g' \rightarrow \g$ maps the subspace $\mathfrak{u'}^-$ to a subspace of $\mathfrak{m}^-:=\mathfrak{m}\cap\mathfrak{u}^-$.
 }
\end{enumerate}
We will refer from now on throughout the paper to a flag domain which satisfies the conditions of the classification problem (*) as a flag domain of Hodge type. Satake studied this problem in the weight $1$ case in connection to the study of (algebraic) families of abelian varieties. In this case $D$ can be identified with an irreducible Hermitian symmetric space of non-compact type (C), i.e. a flag domain in a symplectic isotropic Grassmannian manifold and the embedding is totally geodesic. In his paper \cite{Satake}, Satake gives a complete solution to the problem in the weight $1$ case. 
It turns out that there are very few such embeddings and $D'$ can only be a classical Hermitian symmetric space with no exceptional cases appearing.  
\newline
\noindent
\newline
Recent work of Green, Griffiths and Kerr and the so called \textit{structure theorem} proved in \cite{MTbook} motivates the study of Mumford-Tate domains which are generalizations of period domains and their realization as closed homogeneous submanifolds of a fixed period domain. Let $N_0$ be the group of real points of the Mumford-Tate group (itself a $\Q$-algebraic group) defined at the base point $z_0$, considered as a real Lie group and $N$ its complexification. Then $N_0$ is a group of Hodge type and by definition the orbit $N_0.z_0$ in $D$ is a Mumford-Tate domain. We summarize here some of the basic properties of $N_0$ and $N_0.z_0$ as found in \cite{MTbook}:
\begin{enumerate}
\item{$N_0$ and $N$ respectively are closed Lie subgroups of $G_0$ and $G$ respectively such that $T_0\subset N_0$;}
\item{$N_0.z_0=N_0/(N_0 \cap P_0)$ is a flag domain of its compact dual $N/(N\cap P)$ both embedded as closed submanifolds of $D$ and $Z$ respectively;}
\item{The Lie algebra $\mathfrak{n}$ of $N$ inherits a Hodge structure of weight $0$ from the Hodge structure on $\g$ such that $\mathfrak{n}^{-f,f}:=\mathfrak{n}\cap \g^{-f,f}$ and a Cartan decomposition $\mathfrak{n}=(\mathfrak{n}\cap \k)\oplus (\mathfrak{n}\cap \mathfrak{m})$.}
\item{The holomorphic tangent space of $N_0.z_0$ at $z_0$ is identified with $$\mathfrak{n}\cap \mathfrak{u}^-=\mathfrak{n}\cap \oplus_{f=1}^n\g^{-f,f}=\oplus_{f=1}^n\mathfrak{n}^{-f,f}.$$}
\end{enumerate}
Thus under the extra assumption that $N_0$ is a real semisimple Lie group and $\mathfrak{n}\cap \k=\mathfrak{n}\cap \g^{0,0}$ one has a flag domain of Hodge type as considered in the classification problem (*). Of course in general a Mumford-Tate group might contain compact factors (which turn out to be subgroups of $K_0$) or it can itself be compact or even equal to $K_0$, a torus or the whole group $G$, see \cite{MTbook} for examples. Thus in general the canonically defined base cycle of a Mumford-Tate domain can have higher dimensional intersection with the base cycle $C_0$ of the period domain and thus can be a non-trivial closed submanifold of $C_0$. 
\newline
\noindent
\newline
One of the most important classifying spaces for variations of Hodge structure are those that arise from geometry, related to the study of (algebraic) families of varieties, as target spaces for the image of the period map. Any complex submanifold $N$ which lies in the image of a period map must satisfy an additional condition, namely Griffiths transversality, or in other words the infinitesimal period relation (IPR). Assume $z_0 \in N\subset D$ is a base point. Then the holomorphic tangent space $\mathfrak{a}$ of $N$ at $z_0$ can be identified with a subspace of the tangent space of $D$ at $z_0$ which in turn can be identified with $\mathfrak{u}^-$. Then the infinitesimal period relations says that $\mathfrak{a}$ must be a subspace of $\g^{-1,1}$.
 For an introduction to period mappings and period domains see \cite{Stefan}. 
\newline
\noindent
\newline
Thus a natural first generalization to Satake's paper is that of classifying all proper, equivariant embeddings $f:D' \rightarrow D$, for a flag domain of Hodge type $D'$ satisfying the infinitesimal period relation \footnote{Contrary to the exposition in the introduction throughout the paper we first treat the general classification problem (*) and then explicitly describe the case when the condition $2.$ is replaced by the infinitesimal period relation.}. This is done in \textit{Theorem} \ref{maintheorem} and related corollaries. To solve this problem we first describe certain proper, equivariant embeddings $i:D''\hookrightarrow D$ for $D''$ a product of irreducible Hermitian symmetric spaces (of non-compact type) of type $(A)$ or $(C)$, by explicitly identifying $D''$ as a closed submanifold of $D$. This is done by showing that, with one exception in type (B), $\g^{-1,1}$ is a direct sum of unipotent pieces, each of them being naturally identified with the tangent space to a Grassmannian submanifold of Z of type (A) or (C). This is done in \textit{Proposition} \ref{mainproposition}, where the Harish-Chandra coordinates of each such unipotent piece are made explicit. As a corollary, one obtains a dimension formula for $\g^{-1,1}$ in terms of Hodge numbers (similar formulas are also described for the dimension of $\g^{-p,p}$ with $p$ odd), as a sum of the dimensions of the various Grassmannians involved in the decomposition. This is a known formula, even though not in the present interpretation as a sum of dimensions of various Grassmannians. Furthermore, for each abelian subspace $\mathfrak{a}$ of $\g^{-1,1}$, we construct an abelian subalgebra $\g^{-1,1}_{\mathfrak{a}}$ of $\g$, minimal with respect to the decomposition of $\g^{-1,1}$ as a direct sum of unipotent pieces, such that $\g^{-1,1}_{\mathfrak{a}}$ is the tangent space to an embedded product $G_1\times G_2 \times \dots \times G_s$ of Grassmannians of type (A) (together with one Grassmannian of type (C) for the symplectic case), with the corresponding Hermitian symmetric space of non-compact type identified as a closed submanifold of $D$. Furthermore, we make explicit the Harish-Chandra coordinates of each such piece. Finally, we show that any proper equivariant embedding $f:D'\rightarrow D$ satisfying the infinitesimal period relation must factor through some $i:D''\hookrightarrow D$ and thus prove that $D'$ must be a Hermitian symmetric space embedded as a closed submanifold of a Hermitian symmetric space whose irreducible factors are of type $(A)$ or $(C)$. 
\newline
\noindent
\newline
In order to tackle the general problem we introduce Hodge triples and Hodge brackets, see \textit{definition} \ref{Hodgetriple}, and give a refinement of the Hodge decomposition by describing the direct sum $\g^{-f,f}\oplus\g^{0,0}\oplus \g^{f,-f}$ as a direct sum of Hodge triples. Hodge triples are in some sense a generalization of $\mathfrak{sl}(2)$-triples and are related to embeddings of $\mathfrak{sl}(h,\C)$ in $\g$ for some $h$ depending on the Hodge numbers and $f$. 
\newline
\noindent
\newline
In \textit{section} \ref{results} we give a complete solution to the classification problem (*). The explicit construction done for $\g^{-1,1}$ has a similar equivalent construction for any $\g^{-p,p}$-piece for $p$ odd (and thus for $\mathfrak{m}$) and using the concept of a Hodge triple we show that each $\g^{-p,p}$, $p$ odd and thus $\mathfrak{m}$ itself, with one exception in type $B$, can be written as a direct sum of unipotent pieces each of them being identified with the tangent space to a Grassmannian submanifold of type $(A)$ or $(C)$. As in the case of $\g^{-1,1}$ any such embedding factors through some proper, equivariant embedding $i:D''\hookrightarrow D$ for $D''$ a product of irreducible Hermitian symmetric spaces (of non-compact type) of type $(A)$ or $(C)$, and thus $D'$ must be a Hermitian symmetric space embedded as a closed submanifold of a Hermitian symmetric space (of non-compact type) whose irreducible factors are of type $(A)$ or $(C)$. The main result is contained in the statement of the \textit{Theorem} \ref{teoremaprincipala}.  In \textit{Corollary} \ref{clasicalgroups} we show that in fact the flag domains of Hodge type are Hermitian symmetric space of non-compact type whose irreducible factors are of classical type (A), (B), (C) or (D) and the embeddings describing each such irreducible factor are precisely the embeddings $f: D' \rightarrow D''$ listed by Satake in his classification.\newline
\noindent
\newline
We conclude the paper with a remark explaining the connections of the techniques and results used in the current paper and the work \cite{VHS}, \cite{Robles} and \cite{Kerr1}. Our work here was motivated by our interest in Mumford-Tate domains. For related work in this direction see also \cite{Laza} and \cite{Kerr2}.
\section{Lie structure of period domains}\label{liestructure} 
In this section we introduce the basic notation for the paper. For this we let $G$ denote a complex semisimple Lie group, $G_0$ a real form and $\g$ and $\g_0$ their associated Lie algebras, respectively. Let $P$ be a parabolic subgroup, $\p$ its associated Lie algebra. Denote by $Z:=G/P$ the associated flag manifold and by $D$ a flag domain, i.e. a $G_0$-open orbit in $Z$. This paper is concerned with the flag domains that appear in Hodge theory as classifying spaces for variations of Hodge structure. These are of two flavours, one arising as an open orbit of $SO(p,q)$ acting on a flag manifold associated to $SO(p+q,\C)$ and the second one arising as an open orbit of $Sp(n,\R)$ acting on a flag manifold associated to $Sp(n,\C)$. For reasons which will become clear in the next section one is also interested in the real form $SU(p,q)$ of $SL(p+q,\C)$ and more generally in real semisimple Lie groups of Hodge type. For an introduction to flag domains see \cite{Fels2006}
\newline
\noindent
\newline
The following paragraphs provide a description of the above objects for the classical semisimple Lie groups. Furthermore, a description of their Lie algebras and root systems is given so as to suit the Hodge theoretic considerations of the next chapter. For an introduction to the basic theory of Lie groups and Lie algebras see  \cite{Wallach}, \cite{Knapp2002}. For a basic introduction to the theory of parabolic subgroups see \cite{Bell}.

\subsection{Type (A): $G=SL(l,\C)$}\label{sl}
Fix $(e_1,\dots, e_l)$ the standard ordered basis in $\C^l$ and let $GL(l,\C)$ denote the group of invertible $l\times l$-matrices and $\glc$ its Lie algebra. Then
$$G:=\SLC:=\{g\in GL(l,\C): \det g=1\},$$
$$\g:=\sl:=\{g\in \glc: \tr g=0\}, \quad g=(a_{ij})_{1\le i,j\le l}\in \g,$$
$$\h:=diag[a_{11},\dots,a_{ll}] \text{ Cartan subalgebra},$$
$$\varepsilon_i:\h\rightarrow \C, \, \varepsilon_i(h)=a_{ii}, \, \forall 1\le i \le l,$$
$$\theta:=\{\pm(\varepsilon_i-\varepsilon_j): \, 1\le i<j\le l\},\text{ a set of roots},$$
$$\theta^+:=\{\varepsilon_i-\varepsilon_j:\,1\le i <j \le j\}, \quad  \theta^-:=\{-\varepsilon_i+\varepsilon_j:\,1\le i <j \le j\},$$
$$\Psi:=\{\varepsilon_i-\varepsilon_{i+1}:\,1\le i \le l-1\},$$
a choice of positive roots, negative roots and simple roots respectively. The Dinkyn diagram of $\g$ has $l-1$ nodes corresponding to any choice of $l-1$ simple roots. The root vectors corresponding to positive roots span the upper-triangular matrices which together with $\h$ form a Borel subalgebra $\b^+$ and the root vectors corresponding to negative roots span the lower triangular matrices which together with $\h$ form the opposite Borel subalgebra $\b^-$.  If $(f_1,\dots, f_s)$ is a dimension sequence such that $f_1+\dots+f_s=l$, then the block upper triangular matrices of sizes $f_1$ up to $f_s$ form a parabolic subalgebra $\p^+$ with Levi complement $\mathfrak{s}$ given by block diagonal matrices in $\g$ of sizes $f_1\times f_1$ up to $f_s\times f_s$ and unipotent radical $\mathfrak{u}^+$ such that $\p^+= \mathfrak{u}^+\rtimes \mathfrak{s}$. The block lower triangular matrices of sizes $f_1$ up to $f_s$ form the opposite parabolic subalgebra $\p^-$ with unipotent radical $\mathfrak{u}^-$ and Levi decomposition $\p^-=\mathfrak{u}^-\rtimes s$. Furthermore, $\b^+\subset \p^+$. In general any parabolic subalgebra of $\g$ is characterised by a unique dimension sequence and two parabolic subalgebras in $\g$ have the same dimension sequence if and only if they are conjugate to each other. This can also be expressed by encircling the $\sum_{i=1}^k f_i^{th}$-node for each $1\le k\le s$ in the Dynkin diagram of $\g$. For our choice of simple roots and parabolic subalgebra $\p^+$, this nodes correspond to 
$$\mathcal{N}_{\p}:=\{\varepsilon_{f_1}-\varepsilon_{f_2},\varepsilon_{f_1+f_2}-\varepsilon_{f_1+f_2+1}, \dots, \varepsilon_{f_1+\dots+f_{s-1}}-\varepsilon_{f_1+\dots +f_{s-1}+1}\}.$$
For a more detailed description of this objects and their root structure in general see \cite{Knapp2002}.
\newline
\noindent
\newline
At the group level, recall that given a dimension sequence $f:=(f_1,\dots, f_s)$ with $f_1+\dots+f_s=l$, 
one way to define the flag manifold $Z$ is as the set of all partial flags of type $f$, namely $$\{\mathcal{F}:\, 0\subset F_1\subset\dots\subset F_s=\C^l\},$$
where $\dim(F_i/F_{i-1})=f_i$ and $\dim F_0=0$. Then it is easy to check that $G$ acts transitively on $Z$. By definition a subgroup $P$ of $G$ is called a parabolic subgroup if $P$ is the isotropy subgroup at some base point, a flag, $z_0$ in $Z$. Thus $Z$ can be identified with $G/P$ and furthermore the variety $Z$ parametrizes the set of parabolic subgroups of $G$. For each $i\ge 1$, choose subspaces $S_i$ such that $F_i=F_{i-1}\oplus S_i$. Then by definition the unipotent radical of $P$ is the subgroup $U$ of $P$ consisting of those matrices which induce the identity transformation on $F_i/F_{i-1}$. By definition a Levi complement of $U$ is the subgroup $S$ of $P$ which is the common stabilizer of $S_i$. Furthermore, $P$ can be decomposed as a semidirect product of $U$ and $S$, i.e. $P=U\rtimes S$. Let $P^+$ be a parabolic subgroup with unipotent radical $U^+$ and Levi complement $S$ and $P^-$ denote the opposite parabolic to $P^+$, i.e. $P^-=U^-\rtimes S$ for some unipotent radical $U^-$. Furthermore, denote their Lie algebras with $\p^+$, $\p^-$, $\mathfrak{u}^+$, $\mathfrak{u}^-$, and $\mathfrak{s}$ respectively and recall that $\p^+$ and $\p^-$ define parabolic subalgbras, $\mathfrak{u}^+$, $\mathfrak{u}^-$ their respective unipotent radicals and $\mathfrak{s}$ their Levi complement. If $z_0$ is the base point in $Z$ which defined $P^+$, then the tangent space of $Z$ at $z_0$ can be naturally identified with $\g/\p^+$ which in turn can be naturally identified with $\mathfrak{u}^-$.
\newline
\noindent
\newline
If $(e_1,\dots,e_l)$ is the standard ordered basis in $\C^l$ and $P^+$ the isotropy subgroup of $G$ at the partial flag $\mathcal{E}$ associated to this ordered basis, then $P^+$ is given by block upper triangular matrices of sizes $f_1$, up to $f_s$, respectively and its Lie algebra is $\p^+$, as defined in the first paragraph of this section. If each $f_i$ in the dimension sequence is equal to $1$, then $P^+=B^+$ is the Borel subgroup of upper triangular matrices with Lie algebra $\b^+$.
\newline
\noindent
\newline
The real form of $SL(l,\C)$ of importance in Hodge theory is $SU(a,b)$, with $a+b=l$. This is defined as the group of isometries associated to the standard Hermitian form $h:\C^l\times\C^l\rightarrow \C$, $$((x_1,\dots, x_l),(y_1,\dots,y_l))\mapsto \sum_{i=1}^{a}x_i\overline{y}_i-\sum_{i=a+1}^{l}x_i\overline{y}_i.$$
The flag domains of $SU(a,b)$ are parameterized by the sequences $c:=(c_1,\dots, c_s)$ with $c_i\le a,\,\forall 1\le i \le s$ and $d:=(d_1,\dots, d_s)$ with $d_i\le b,\,\forall 1\le i \le s$ such that $c_i+d_i=f_i,\,\forall 1\le i \le s$. For a fixed pair $(c,d)$, the open orbit $D_{c,d}$ is described by the set of all flags $\{\mathcal{F}:\, 0\subset F_1\subset \dots\subset F_s\}$ such that $F_i$ has signature $(\sum_{j=1}^ic_j,\sum_{j=1}^id_j)$ with respect to $h$, for all $1\le i \le s$.
An equivalent parametrisation of such an orbit is given by a sequence $\alpha$ of "+"'s and "-'"s. The sequence is defined by prescribing $s$ adjacent blocks , where each block $i$ has $c_i$-"+"'s and $d_i$-"-"'s. 

\subsection{Type (D): $G:=SO(2l,\mathbb{C})$}\label{so}
Let $Q:\C^{2l}\times\C^{2l}\rightarrow \C$ be the symmetric bilinear form on $\C^{2l}$ defined by the following matrix: 
$$I:=\begin{pmatrix}
O_l & I_l \\
I_l & O_l
\end{pmatrix},$$
where $I_l$ denotes the $l\times l$ identity matrix and $O_l$ the $l\times l$ zero matrix. 
Then $G$ is defined as the group of isometries in $SL(2l,\C)$, which preserve $Q$, i.e. 
$$G:=SO(2l,\C):=\{g\in SL(2l,\C):\, gIg^t=I\},$$
$$\g:=\soc:=\{g\in \slll: gI+Ig^{t}=0\}.$$
Let $g\in \g$. Then $g$ has the following block-matrix form: 
$$g=\{\begin{pmatrix}
A & B \\
C & -A^{t}
\end{pmatrix}: \, B+B^{t}=0, C+C^{t}=0, \,A,B,C\in\glc\}.$$
\noindent
Furthermore, we choose the following Cartan subalgebra, the set of roots with respect to it and the following positive roots, negative roots and simple roots respectively:
$$\h:=diag(a_{11},\dots, a_{ll},-a_{11},\dots, -a_{ll}),$$  
$$\theta:=\{\pm \varepsilon_i\pm \varepsilon_j:\,1\le i<j\le l\},$$   
$$\theta^+:=\{\varepsilon_i\pm \varepsilon_j:\,1\le i<j\le l\},$$ 
$$\theta^-:=\{-\varepsilon_i\pm \varepsilon_j:\,1\le i<j\le l\},$$
$$\Psi:=\{\varepsilon_i-\varepsilon_{i+1}:\,1\le i \le l-1 \}\cup\{\varepsilon_{l-1}+\varepsilon_{l}\}.$$
Consider the natural, diagonal embedding of $\sl$ in $\soc$ (with the choice of Cartan subalgebra and root structure described in \ref{sl}), given by $$A\mapsto \diag[A,-A^t],\,\forall A\in \sl.$$
The root vectors which span this $\sl$ correspond to the following set of roots: $\theta_s^+:=\{\varepsilon_i-\varepsilon_j:1\le i<j\le l\}$, with their respective root vectors spanning the upper triangular matrices in $A$, while the negative roots are given by $\theta_s^-:=\{-\varepsilon_i+\varepsilon_j:1\le i<j\le l\}$, with their respective root vectors spanning the lower triangular matrices in $A$. Let $\mathcal{D}$ be the Dynkin diagram of $\g$ and think of $\Psi$ as a set of simple roots representing its nodes. Then this embedding is equivalent to the choice of subdiagram $\mathcal{D}'$ of $\mathcal{D}$ corresponding to the subset $\Psi':=\Psi-\{\varepsilon_{l-1}+\varepsilon_{l}\}$ of simple roots.  
\noindent
The antidiagonal block is composed of two antisymmetric matrices $$g=\{\begin{pmatrix}
0 & B \\
C & 0
\end{pmatrix}: \, B+B^{t}=0, C+C^{t}=0, \,\in\glc\}.$$
The root vectors corresponding to the set of positive roots $\{\varepsilon_i+\varepsilon_j:1\le i<j\le m\}$ span the $B$-part, while the  root vectors corresponding to the set of negative roots $\{-\varepsilon_i-\varepsilon_j:1\le i<j\le m\}$ span the $C$-part. 
\newline
\noindent
\newline
Another important embedding of $\sl$ (with the choice of Cartan subalgebra and root structure described in \ref{sl}) in $\soc$ is given as follows. Let $(d,l-d)$ be a dimension sequence defining a maximal parabolic subalgebra $\mathfrak{p}^+$ of $\sl$. Then the Levi component $\mathfrak{s}$ of $\mathfrak{p}^+$ is mapped to the spann of the set of root vectors corresponding to the following set of roots: $\{\pm (\varepsilon_i-\varepsilon_j): \, 1\le i \le d-1,\, i<j\le d\} \cup \{\pm (\varepsilon_i-\varepsilon_j): \, d+1\le i \le l-1,\, i<j\le l\}$ together with the diagonal Cartan subalgebra $\h$. The positive unipotent part $\mathfrak{u}^+$ of $\mathfrak{p}^+$ is mapped to the span of the set of root vectors corresponding to the following set of roots: $\{ (\varepsilon_i-\varepsilon_j): \, 1\le i \le d,\, d+1<j\le l\}$. The negative unipotent part $\mathfrak{u}^-$ is mapped to the root vectors in $\soc$ corresponding to the following set of roots $\{ (\varepsilon_i+\varepsilon_j): \, 1\le i \le d-1,\, i<j\le d\}$. This embedding is equivalent to the choice of subdiagram $\mathcal{D}'$ of $\mathcal{D}$ corresponding to the subset $\Psi':=\Psi-\{\varepsilon_{d}+\varepsilon_{d+1}\}$ of simple roots. From now on call this embedding the anti-diagonal embedding. 
\newline
\noindent
\newline
 Let $f=(f_1,\dots, f_s)$ be a dimension sequence. Define the flag manifold of isotropic flags with respect to $Q$ as
 $$Z:=\{\mathcal{F}: \text{ either } F_i\subset F^{\perp}_j \text{ or } F^{\perp}_j\subset F_i,\,\forall 1\le i,j\le s\},$$
 where the orthogonal complement is taken with respect to $Q$. One can then check that $G$ acts transitively on $Z$ and the isotropy subgroup of $G$ and any base point $z\in Z$ is a parabolic subgroup $P_z$. As for all flag manifolds, the variety $Z$ parametrizes the set of parabolic subgroups corresponding to a fixed given dimension sequence $f$. Let $\tilde{z}_0$ be the full isotropic flag associated to the ordered basis $\{e_1,\dots, e_l,e_{2l},\dots, e_{l}\}$ of $\C^{2l}$ and let $B_{\tilde{z}_0}$ denoted the isotropy subgroup of $G$ at $\tilde{z}_0$. Then it is immediate to check that $B_{\tilde{z}_0}$ is the Borel subgroup in $G$ with Lie algebra $\b_{\tilde{z}_0}=\h\bigoplus\oplus_{\alpha \in \theta^+}\g_{\alpha}$, the Borel subalgebra given by our choice of positive roots. Without confusion let $\tilde{z}_0$  also denote the partial isotropic flag in $Z$ corresponding to the dimension sequence $f$ and the ordered basis $\{e_1,\dots, e_l,e_{2l},\dots, e_{l}\}$. Then $B_{\tilde{z}_0}\subset P_{\tilde{z}_0}$ and one has a similar description of $\p_{\tilde{z_0}}$, the Lie algebra of $P_{\tilde{z_0}}$ in terms of simple roots and Levi decomposition as in section \ref{sl}. 
 \newline
 \noindent
 \newline
The real form of $G=SO(2l,\C)$ of importance in Hodge theory is $SO(2a,2b)$, with $2a+2b=2l$. This is defined as the group of isometries in $SO(2l,\C)$ of the following Hermitian symmetric form on $\C^{2l}$.
\newline
\noindent
\newline
Let $\sigma:\C^{2l}\rightarrow \C^{2l}$ be a complex conjugation on $\C^{2l}$,i.e. an antilinear involution, defined by $$(x_1,\dots,x_l,x_{l+1}, \dots,x_{2l})\mapsto (\overline{x}_{l+1},\dots,\overline{x}_{2l},\overline{x}_1,\dots, \overline{x}_l).$$
Let $h:\C^{2l}\times\C^{2l}\rightarrow \C$ be defined by 
$$((x_1,\dots,x_{2l}),(y_1,\dots,y_{2l}))\mapsto \sum_{i=1}^ax_i\overline{y}_i-\sum_{i=a+1}^{l}x_i\overline{y}_i+\sum_{i=l+1}^{l+a}x_i\overline{y}_i-\sum_{i=l+a+1}^{2l}x_i\overline{y}_i.$$
Then $\sigma$ acts on the standard basis by sending $e_i\mapsto e_{l+i}$, for all $1\le i \le  l$ and one obtains a decomposition of $\C^{2l}$ as a direct sum of its positive and negative subspaces, i.e. $$\C^{2l}=E^+\oplus E^-\oplus \overline{E}^+\oplus \overline{E}^-,$$ where $E^+:=<e_1,\dots, e_a>$, $\overline{E}^{+}=<e_{l+1},\dots, e_{l+a}>$ and $E^-:=<e_{a+1},\dots, e_{l}>$, $\overline{E}^-=<e_{l+a+1},\dots, e_{2l}>$.
\newline
\noindent
\newline
As in the previous case the flag domains are parametrized by signature, a sequence of "+"'s and "-"'s 
$$D:=\{\mathcal{F}:\, \mathcal{F} \text{ is isotropic and signature of } F_i \text{ is } (c_i,d_i) \,\forall 1\le i \le s\},$$
\noindent
From a Hodge theoretic point of view an important case is that of isotropic flag manifolds associated to a symmetric dimension sequence $$(f_1,\dots,f_{m-1}, f_m,f_{m-1}\dots, f_1),$$ with $f_m$ even, such that the corresponding signature sequence parametrising a flag domain has constant sign in each of its blocks. In particular the flag domains have compact isotropy and are measurable. Observe that $h$ restricts to a Hermitian form on $E^+\oplus E^-\cong \C^{l}$ and defines the isometry group 
$SU(a,b)$ of $SL(l,\C)$. The Lie algebra $su(a,b)$ of $\sl$ embeds diagonally inside $\soc$. There is thus a one-to-one correspondence in between flag domains of $SU(a,b)$ with constant sign on each block of a signature sequence, and flag domains of $SO(2a,2b)$ defined by a symmetric dimension sequence with constant sign on each block of the signature sequence. If $\alpha$ is a signature sequence with constant sign on each block parametrising an SU(a,b) flag domain and $\alpha'$ is the same sequence read in reverse, then $\alpha\alpha'$ parametrises the corresponding $SO(2a,2b)$-flag domain. 
\newline
\noindent
\newline
 In general any parabolic subalgebra of $\g$ is characterised by a unique dimension sequence and two parabolic subalgebras in $\g$ have the same dimension sequence if and only if they are conjugate to each other. This can also be expressed by encircling the $\sum_{i=1}^k f_i^{th}$-node for each $1\le k\le s$ in the Dynkin diagram of $\g$. For our choice of simple roots and parabolic subalgebra $\p_{\tilde{z}_0}$ corresponding to a symmetric dimension sequence, this nodes correspond to 
$$\mathcal{N}_{\p}:=\{\varepsilon_{f_1}-\varepsilon_{f_2},\varepsilon_{f_1+f_2}-\varepsilon_{f_1+f_2+1}, \dots, \varepsilon_{f_1+\dots+f_{m-1}}-\varepsilon_{f_1+\dots +f_{m-1}+1}\}.$$

\subsection{Type (B): $G:=SO(2l+1,\mathbb{C})$}\label{soo}
Let $Q:\C^{2l+1}\times\C^{2l+1}\rightarrow \C$ be the symmetric bilinear form on $\C^{2l+1}$ defined by the following matrix: 
$$I:=\begin{pmatrix}
O_l & O & I_l \\
O & 1 & O \\
I_l & O &O_l
\end{pmatrix},$$
where $I_l$ denotes the $l\times l$ identity matrix. 
Then $G$ is defined as the group of isometries in $SL(2l+1,\C)$, which preserve $Q$, i.e. 
$$G:=SO(2l+1,\C):=\{g\in SL(2l+1,\C):\, gIg^t=I\},$$
$$\g:=\socc:=\{g\in \slll: gI+Ig^{t}=0\}.$$
Let $g\in \g$. Then $g$ has the following block-matrix form: 
$$g=\{\begin{pmatrix}
A & x & B \\
y^t & 0 & -x^t\\
C & -y & -A^{t}
\end{pmatrix}: \, B+B^{t}=0, C+C^{t}=0, \,A,B,C\in\glc, \, x,y\in \C^l\}.$$
\noindent
We make the following choice of Cartan subalgebra, root system with respect to it, positive roots, negative roots and simple roots respectively:
$$\h:=diag(a_{11},\dots, a_{ll},0, -a_{11},\dots, -a_{ll}),$$  
$$\theta:=\{\pm \varepsilon_i\pm \varepsilon_j:1\le i<j\le l\} \cup \{\pm \varepsilon_i: 1\le i \le l\},$$   
$$\theta^+:=\{\varepsilon_i\pm \varepsilon_j:1\le i<j\le l\}\cup \{ \varepsilon_i: 1\le i \le l\},$$ 
$$\theta^-:=\{-\varepsilon_i\pm \varepsilon_j:1\le i<j\le l\}\cup \{-\varepsilon_i:1\le i \le j\}.$$
$$\Psi:=\{\varepsilon_i-\varepsilon_{i+1}:\,1\le i \le l-1 \}\cup\{\varepsilon_{l}\}.$$
As in the even dimensional case, consider the diagonal and anti-diagonal $\sl$-embeddings. Furthermore, the negative short roots span the $y$-block, and the positive short roots the $x$-block.  
\newline
\noindent
\newline
 Let $f=(f_1,\dots, f_s)$ be a dimension sequence. Define the flag manifold of isotropic flags with respect to $Q$ as
 $$Z:=\{\mathcal{F}: \text{ either } F_i\subset F^{\perp}_j \text{ or } F^{\perp}_j\subset F_i,\,\forall 1\le i,j\le s\},$$
 where the orthogonal complement is taken with respect to $Q$. One can then check that $G$ acts transitively on $Z$ and the isotropy subgroup of $G$ and any base point $z\in Z$ is a parabolic subgroup $P_z$. As for all flag manifolds, the variety $Z$ parametrizes the set of parabolic subgroups corresponding to a fixed given dimension sequence $f$. Let $\tilde{z}_0$ be the full isotropic flag associated to the ordered basis $\{e_1,\dots, e_l,e_{l+1},e_{2l+1},\dots, e_{l}\}$ of $\C^{2l+1}$ and let $B_{\tilde{z}_0}$ denoted the isotropy subgroup of $G$ at $\tilde{z}_0$. Then it is immediate to check that $B_{\tilde{z}_0}$ is the Borel subgroup in $G$ with Lie algebra $\b_{\tilde{z}_0}=\h\bigoplus\oplus_{\alpha \in \theta^+}\g_{\alpha}$, the Borel subalgebra given by our choice of positive roots. Without confusion let $\tilde{z}_0$  also denote the partial isotropic flag in $Z$ corresponding to the dimension sequence $f$ and the ordered basis $\{e_1,\dots, e_l,e_{l+1},e_{2l+1},\dots, e_{l}\}$. Then $B_{\tilde{z}_0}\subset P_{\tilde{z}_0}$ and one has a similar description of $\p_{\tilde{z_0}}$, the Lie algebra of $P_{\tilde{z_0}}$ in terms of simple roots and Levi decomposition as in section \ref{sl}. 
\newline
 \noindent
 \newline
The real form of $G=SO(2l+1,\C)$ of importance in Hodge theory is $SO(2a,2b+1)$, with $2a+2b=2l$. This is defined as the group of isometries in $SO(2l+1,\C)$ of the following Hermitian symmetric form on $\C^{2l+1}$.
\newline
\noindent
\newline
Let $\sigma:\C^{2l+1}\rightarrow \C^{2l+1}$ be a complex conjugation on $\C^{2l+1}$,i.e. an antilinear involution, defined by $$(x_1,\dots,x_l,x_{l+1},x_{l+2} \dots,x_{2l+1})\mapsto (\overline{x}_{l+2},\dots,\overline{x}_{2l+1},\overline{x}_{l+1},\overline{x}_1,\dots, \overline{x}_l).$$
Let $h:\C^{2l+1}\times\C^{2l+1}\rightarrow \C$ be defined by 
$$((x_1,\dots,x_{2l}),(y_1,\dots,y_{2l}))\mapsto \sum_{i=1}^ax_i\overline{y}_i-\sum_{i=a+1}^{l+1}x_i\overline{y}_i+\sum_{i=l+2}^{l+a}x_i\overline{y}_i-\sum_{i=l+a+2}^{2l}x_i\overline{y}_i.$$
Then $\sigma$ acts on the standard basis by sending $e_i\mapsto e_{l+i+1}$, for all $1\le i \le  l$ and $e_{l+1}\mapsto e_{l+1}$. One obtains a decomposition of $\C^{2l+1}$ as a direct sum of its positive and negative subspaces, i.e. $$\C^{2l+1}=E^+\oplus E^-\oplus<e_{l+1}> \oplus\overline{E}^+\oplus \overline{E}^-,$$ where $E^+:=<e_1,\dots, e_a>$, $\overline{E}^{+}=<e_{l+2},\dots, e_{l+a}>$ and $E^-:=<e_{a+1},\dots, e_{l}>$, $\overline{E}^-=<e_{l+a+2},\dots, e_{2l}>$.
\newline
\noindent
\newline
As in the previous case the flag domains are parametrized by signature, see \ref{so} for details and the parabolic subgroups to be considered are those defined by a symmetric dimension sequence $(f_1,f_2,\dots, f_{m-1}, f_m, f_{m-1},\dots, f_2,f_1)$, with $f_m$ odd in this case.
 \newline
\noindent
\newline
Then any parabolic subalgebra of $\g$ can be visualized by encircling the $\sum_{i=1}^k f_i^{th}$-node for each $1\le k\le m-1$ in the Dynkin diagram of $\g$. For our choice of simple roots and parabolic subalgebra $\p_{\tilde{z}_0}$ corresponding to a symmetric dimension sequence, this nodes correspond to 
$$\mathcal{N}_{\p}:=\{\varepsilon_{f_1}-\varepsilon_{f_2},\varepsilon_{f_1+f_2}-\varepsilon_{f_1+f_2+1}, \dots, \varepsilon_{f_1+\dots+f_{m-1}}-\varepsilon_{f_1+\dots +f_{m-1}+1}\}.$$
\noindent
\subsection{Type (C): $G:=Sp(2l,\mathbb{C})$}\label{sp}
Let $Q:\C^{2l}\times\C^{2l}\rightarrow \C$ be the antisymmetric bilinear form on $\C^{2l}$ defined by the following matrix: 
$$J:=\begin{pmatrix}
O_l & J_l \\
-J_l & O_l
\end{pmatrix},$$
where $J_l$ denotes the $l\times l$ identity matrix. 
Then $G$ is defined as the group of isometries in $SL(2l,\C)$, which preserve $Q$, i.e. 
$$G:=Sp(2l,\C):=\{g\in SL(2l,\C):\, gJg^t=J\},$$
$$\g:=\spc:=\{g\in \slll: gJ+Jg^{t}=0\}.$$
Let $g\in \g$. Then $g$ has the following block-matrix form: 
$$g=\{\begin{pmatrix}
A & B \\
C & -A^{t}
\end{pmatrix}: \, B-B^{t}=0, C-C^{t}=0, \,A,B,C\in\glc\}.$$
\noindent
We make the follwing choice of Cartan aubalgebra, root system, positive roots, negative roots and simple roots respectively:
$$\h:=diag(a_{11},\dots, a_{ll},-a_{11},\dots, -a_{ll}),$$  
$$\theta:=\{\pm \varepsilon_i\pm \varepsilon_j:1\le i<j\le l\} \cup \{\pm 2\varepsilon_i: 1\le i \le l\},$$   
$$\theta^+:=\{\varepsilon_i\pm e_j:1\le i<j\le l\}\cup \{ 2\varepsilon_i: 1\le i \le l\},$$ 
$$\theta^-:=\{-\varepsilon_i\pm \varepsilon_j:1\le i<j\le l\}\cup \{-2\varepsilon_i:1\le i \le j\}.$$
$$\Psi:=\{\varepsilon_i-\varepsilon_{i+1}:\,1\le i \le l-1 \}\cup\{2\varepsilon_{l}\}.$$
\noindent
As for the orthogonal groups, consider the diagonal and anti-diagonal $\sl$-embeddings. Furthermore, the positive long roots span the $B$-diagonal, and the negative long roots span the $C$-diagonal.  
\newline
\noindent
\newline
 Let $f=(f_1,\dots, f_s)$ be a dimension sequence. Define the flag manifold of isotropic flags with respect to $Q$ as
 $$Z:=\{\mathcal{F}: \text{ either } F_i\subset F^{\perp}_j \text{ or } F^{\perp}_j\subset F_i,\,\forall 1\le i,j\le s\},$$
 where the orthogonal complement is taken with respect to $Q$. One can then check that $G$ acts transitively on $Z$ and the isotropy subgroup of $G$ and any base point $z\in Z$ is a parabolic subgroup $P_z$. As for all flag manifolds, the variety $Z$ parametrizes the set of parabolic subgroups corresponding to a fixed given dimension sequence $f$. Let $\tilde{z}_0$ be the full isotropic flag associated to the ordered basis $\{e_1,\dots, e_l,e_{2l},\dots, e_{l}\}$ of $\C^{2l}$ and let $B_{\tilde{z}_0}$ denoted the isotropy subgroup of $G$ at $\tilde{z}_0$. Then it is immediate to check that $B_{\tilde{z}_0}$ is the Borel subgroup in $G$ with Lie algebra $\b_{\tilde{z}_0}=\h\bigoplus\oplus_{\alpha \in \theta^+}\g_{\alpha}$, the Borel subalgebra given by our choice of positive roots. Without confusion let $\tilde{z}_0$  also denote the partial isotropic flag in $Z$ corresponding to the dimension sequence $f$ and the ordered basis $\{e_1,\dots, e_l,e_{2l},\dots, e_{l}\}$. Then $B_{\tilde{z}_0}\subset P_{\tilde{z}_0}$ and one has a similar description of $\p_{\tilde{z_0}}$, the Lie algebra of $P_{\tilde{z_0}}$ in terms of simple roots and Levi decomposition as in section \ref{sl}. 
\newline
 \noindent
 \newline
The real form of $G=Sp(2l,\C)$ of importance in Hodge theory is $Sp(a+b,b+a)$, with $a+b=l$. This is defined as the group of isometries in $Sp(2l,\C)$ of the following Hermitian symmetric form on $\C^{2l}$.
\newline
\noindent
\newline
Let $\sigma:\C^{2l}\rightarrow \C^{2l}$ be a complex conjugation on $\C^{2l}$,i.e. an antilinear involution, defined by $$(x_1,\dots,x_l,x_{l+1} \dots,x_{2l})\mapsto (\overline{x}_{l+1},\dots,\overline{x}_{2l},\overline{x}_1,\dots, \overline{x}_l).$$
Let $h:\C^{2l}\times\C^{2l}\rightarrow \C$ be defined by 
$$((x_1,\dots,x_{2l}),(y_1,\dots,y_{2l}))\mapsto -\sum_{i=1}^ax_i\overline{y}_i+\sum_{i=a+1}^{l}x_i\overline{y}_i+\sum_{i=l+1}^{l+a}x_i\overline{y}_i-\sum_{i=l+a+1}^{2l}x_i\overline{y}_i.$$
Then $\sigma$ acts on the standard basis by sending $e_i\mapsto e_{l+i}$, for all $1\le i \le  l$. One obtains a decomposition of $\C^{2l}$ as a direct sum of its positive and negative subspaces, i.e. $$\C^{2l}=E^-\oplus E^+\oplus\overline{E}^-\oplus \overline{E}^+,$$ where $E^-:=<e_1,\dots, e_a>$, $\overline{E}^{-}=<e_{l+1},\dots, e_{l+a}>$ and $E^+:=<e_{a+1},\dots, e_{l}>$, $\overline{E}^+=<e_{l+a},\dots, e_{2l}>$. Note that $\overline{E}^-$ is a positive subspace and $\overline{E}^+$ is a negative subspace. 
\newline
\noindent
\newline
As in the previous case the flag domains are parametrized by signature, see \ref{so} for details and the parabolic subgroups to be considered are those defined by a symmetric dimension sequence $(f_1,f_2,\dots, f_{m+1}, f_{m+1},\dots, f_2,f_1)$.
 \newline
\noindent
\newline
Then any parabolic subalgebra of $\g$ can be visualized by encircling the $\sum_{i=1}^k f_i^{th}$-node for each $1\le k\le m+1$ in the Dynkin diagram of $\g$. For our choice of simple roots and parabolic subalgebra $\p_{\tilde{z}_0}$ corresponding to a symmetric dimension sequence, this nodes correspond to 
$$\mathcal{N}_{\p}:=\{\varepsilon_{f_1}-\varepsilon_{f_2},\varepsilon_{f_1+f_2}-\varepsilon_{f_1+f_2+1}, \dots, \varepsilon_{f_1+\dots+f_{m+1}}-\varepsilon_{f_1+\dots +f_{m+1}+1}\}.$$
\noindent

\section{Main results}
This section is devoted to stating and proving the main results of the paper. Since our work here is somewhat technical we begin with an outline of its structure. \textit{Section} \ref{variations} is a basic introduction to variations of Hodge structures, the definition of period domains and their explicit connection to flag domains as introduced in \textit{Section} \ref{liestructure}. Furthermore, given a base point in a period domain $D\hookrightarrow G/P$, one has a natural way to define a Hodge decomposition of weight $0$ at the Lie algebra level, $\g=\oplus\g^{-p,p}$, which in turn allows one to define a natural involution on the Lie algebra in question.
\newline
\noindent
\newline
\textit{Section} \ref{combinatorics} is concerned with combinatorial structures associated to period domains and it is the starting point to tackle the classification problem (*). This section is divided into three parts each of them being concerned with one type of period domain, even weight and even dimension, odd weight and even dimension and even weight and odd dimension respectively, i.e. flag domains in flag manifolds of $SO(2l,\C)$, $Sp(2l,\C)$ and $SO(2l+1,\C)$ respectively. Each part begins with an explicit description of the flag domain $D\hookrightarrow G/P$ in question together with a combinatorial rule used to choose a special base point $z_0\in D$. Then one looks at the Lie algebra $\g$ of $G$ and analyses its action on $z_0$ via endomorphisms. This produces a block decomposition of any $g\in \g$ easily described with the help of the combinatorial rule used to define $z_0$. Finally, we give a refinement of the Hodge decomposition on $\g$ and represent each $\g^{-p,p}$ piece as a direct sum of blocks. This is done in \textit{Proposition} \ref{orthogonaleven}, \textit{Proposition} \ref{symplectic} and \textit{Proposition} \ref{orthogonalodd} respectively. Furthermore, at the end of each subsection we give a \textit{description of the canonical base cycle} associated to the choice base point $z_0$ together with three \textit{complete examples} for a weight $8$ Hodge structure on a $20$-dimensional vector space, a weight $5$ Hodge structure on a $14$-dimensional vector space and finally a weight $4$-Hodge structure on an $11$-dimensional vector space. 
The three parts of this section are very similar to each other. However, in order not to confuse the readers with indices differing from one Lie group type to the other, we chose to treat each case separately. 
\newline
\noindent
\newline
In \textit{Section} \ref{combinatorics} it is thus shown that certain "blocks" are the building structures for elements of $\g^{-p,p}$ viewed as linear algebraic objects. In \textit{Section} \ref{results} we show that in fact this blocks carry an additional structure and can be viewed as unipotent radicals for given Lie algebra embeddings of $\mathfrak{sl}$, $\mathfrak{so}$ or $\mathfrak{sp}$, respectively in $\g$. The data required to describe such an embedding is encoded in a Hodge triple which is defined in \textit{Definition} \ref{Hodgetriple} and further explained in the paragraph thereafter. \textit{Theorem} \ref{hodgetriple} is the first main result of this paper making explicit this correspondence which is also treated in detail in an example for a weight $4$ Hodge structure on a $12$-dimensional vector space. The basic properties of Hodge triples and their associated Hodge brackets are described in \textit{Lemma} \ref{primalema}, \textit{Lemma} \ref{comutativitate}, \textit{Lemma} \ref{bracketcompact} and \textit{Lemma} \ref{ungrassmann}. The main result of this paper, giving a complete solution to the classification problem (*) is \textit{Theorem} \ref{teoremaprincipala} which shows that a flag domain of Hodge type is a Hermitian symmetric space of non-compact type. In \textit{Corollary} \ref{clasicalgroups} we show that in fact the flag domains of Hodge type are Hermitian symmetric space of non-compact type whose irreducible factors are of classical type (A), (B), (C) or (D) and the embeddings describing each such irreducible factor are precisely the embeddings $f: D' \rightarrow D''$ listed by Satake in his classification.
\newline
\noindent
\newline
Hereafter, we are concerned with special cases of \textit{Theorem} \ref{teoremaprincipala} in the case of an embedding satisfying the infinitesimal period relation and certain combinatorial description of objects involved. \textit{Proposition} \ref{mainproposition} restates the general results concerning the $\g^{-p,p}$ block characterisation for the subspace $\g^{-1,1}$ to explicitely point out the identification of each of its blocks with a tangent space to a certain Grassmannian submanifold, i.e. with a unipotent radical, and recall the explicit description of the Harisch-Chandra coordinates characterizing each block. \textit{Corollary} \ref{dimensionformula} and \textit{Corollary} \ref{generaldimensionformula} give a dimension formula for $\g^{-1,1}$ and for $\g^{-p,p}$, $p$ odd, as a sum of the dimension of the various Grassmannians involved in their block decomposition. 
\newline
\noindent
\newline
In the \textit{Subsection} \textbf{Abelian subspaces of $\g^{-1,1}$} of \textit{Section} \ref{results}, we give a procedure for constructing abelian subspaces of $\g^{-1,1}$ following combinatorially defined objects called Hodge paths and Hodge sequences. Then \textit{Theorem} \ref{maintheorem}, a special case of \textit{Theorem} \ref{teoremaprincipala}, states that any abelian subspace ${\mathfrak{a}}$ of $\g^{-1,1}$ is contained in a bigger abelian subspace $\g^{-1,1}_{\mathfrak{a}}\subset \g^{-1,1}$ constructed using a Hodge path and a Hodge sequence depending on $\a$. Furthermore, with one exception in type (B), the subalgebra $\g^{-1,1}_{\mathfrak{a}}$ can be naturally identified with the tangent space at a base point in $D$ to an equivariently-embedded product of Grassmannian submanifolds $G_1\times G_2\times \dots \times G_s$. For even weight each $G_i$ is of type (A), whereas for odd weight each $G_i$ is of type (A) with one possible member of type (C).
The dimension of $\g_{\mathfrak{a}}^{-1,1}$ is easily computed as a sum of the dimensions of the various Grassmannians and thus provides an upper bound for the dimension of $\mathfrak{a}$. In the \textit{Subsection} \textbf{Abelian subspaces of $\g^{-p,p}$} we carry out similar concrete combinatorial descriptions for general $p$ odd.

\subsection{Variations of Hodge structure}\label{variations}
Let $V$ denote a finite dimensional $\R$- (or $\Q$-) vector space and $V_{\C}:=V\otimes_{\R} \C$, and $\sigma$ the associated complex conjugation on $V_\C$. If $W\subset V_\C$, then $\overline{W}:=\sigma(W)$. Furthermore, let $p,q,n\in \Z$. A \textbf{Hodge decomposition} of \textit{weight} $n$ and \textit{type} $(p,q)$ on $V$ is a decomposition of $V_\C$ as $$V_\C=\bigoplus_{p+q=n}V^{p,q},$$ such that $V^{q,p}=\overline{V}^{p,q}$. The numbers $h^{p,q}:=\dim{V^{p,q}}$ are called \textbf{Hodge numbers}.\newline
\noindent
Given a Hodge decomposition one obtains a \textbf{Hodge filtration}, a partial flag $$\mathcal{F}: 0\subset F^{n}\subset \dots \subset F^{1}\subset F^{0}=V_\C,$$ in $V_\C$, defined by $F^{i}:=\bigoplus_{p=i}^nV^{p,n-p}$ and $F^{i}\oplus \overline{F}^{n-i+1}=V_{\C}$. Note that in this context we prefer to use a decreasing sequence of indices to denote the components of a flag, in contrast with the increasing sequence used in section \ref{liestructure}.
\newline
\noindent
\newline
There are two equivalent definition of a polarised Hodge structure of weight $n$ as follows. A polarised Hodge structure of weight $n$ is given by a Hodge decomposition, or a Hodge filtration together with a bilinear form $Q$ on $V$, symmetric or antisymmetric depending on wether $n$ is even or odd respectively, such that the Hodge-Riemann bilinear relations are satisfied. In terms of the Hodge filtration, these are expressed as: 
$$\begin{cases}
Q(F^p,F^{n-p+1})=0\\
Q(v,C\overline{v}>0),
\end{cases}
$$
where $C$ is the \textit{Weil operator} of multiplication by $i$ defined on $V_{\C}$ by $v\mapsto i^{p-q}v$, for all $v\in V^{p,q}$.
In terms of the Hodge decomposition, the Hodge-Riemann bilinear relations say that the Hodge decomposition is orthogonal relative to the Hermitian form $h$ on $V_{C}$ defined by 
$$\begin{cases}
h(v_1,v_2)=Q(v_1,C\overline{v}_2)\\
h \text{ is positive-definite}.
\end{cases}
$$
Let $\{h^{p,q}:\, p,q\in \Z\}$ be a given set of Hodge numbers for a weight $n$ and a Hodge decomposition on $V$, i.e. $p+q=n$, $\dim V^{p,q}=h^{p,q}$, $h^{p,q}=h^{q,p}$ and $\dim V=\sum_{p=0}^n h^{p,q}$. Then the \textbf{period domain} $D$ associated to the above data is the set of polarised Hodge structures $(V,Q)$ with the given Hodge numbers. 
\newline
\noindent
\newline
Let $f^{i}:=\sum_{p\ge i}h^{p,n-p}$ denote the dimension of $F^{i}$, the $i^\text{th}$-component of the associated Hodge filtration. Then the compact dual is the set of flags $\{\mathcal{F}: 0\subset F^n\subset \dots \subset F^0=V_\C\}$ such that $\dim F^p=f^p$ and which satisfy the first Hodge-Riemann bilinear relation: $Q(F^p,F^{n-p+1})=0$. 
Identify $V_\C$ with $\C^k$ by choosing a basis. 
\newline
\noindent
\newline
Let $2l$ or $2l+1$ denote the dimension of $V$ depending on whether the dimension is even or odd respectively.
The data required to define the period domain and its compact dual is equivalent, in terms of language associated to flag manifolds and flag domains, to the following. The weight $n$ defines the complex simple Lie group $G$ of isometries of $\C^{2l}$ or $\C^{2l+1}$ respectively, which preserve the bilinear form $Q$. In the case of $n$ even one obtains $G=SO(2l,\C)$ or $G=SO(2l+1,\C)$, while in the case of $n$ odd, one obtains $G=Sp(2l,\C)$. 
\newline
\noindent
\newline
Specifying a set of Hodge numbers is equivalent to specifying a symmetric dimension sequence for a parabolic subgroup P. Together with the first Hodge-Riemann bilinear relation this amounts to identifying the compact dual of a period domain with the flag manifold $G/P$ of isotropic flags with respect to the given bilinear form $Q$. The second Hodge-Riemann bilinear relation is an open condition which identifies the period domain with a flag domain $D$. The real form acting on $D$ is given as the group of isometries in $G$ preserving the given Hermitian form $h$ and the signature sequence parametrising the flag domain can be read from the signature of the base Hodge structure.  
\newline
\noindent
\newline
Conversely, given a symmetric dimension sequence $$(f_0,\dots, f_{m-1},f_m,f_{m-1},\dots, f_0)$$  and a sequence of pluses and minuses parametrising a flag domain such that each block has constant sign, one defines inductively a variation of Hodge structure of weight n (to be defined). First, one defines a virtual sequence $f'$ from the initial sequence $f$ as follows: start with the sequence $f'$ being equal to $f$. Then if the block corresponding to $f'_i$ has the same sign as the block corresponding to $f'_{i-1}$, shift the indices of $f'_j$ for all $j\ge i$ by one, and introduce a new $f'_i$ which is set to be zero. Then the weight $n$ is given by $2m$ plus $2\times$(number of shifts) and $h^{n-i,i}:=f'_i$. If at the end of the process the sequence $f'$ equals the sequence $f$, then we have a Hodge structure with no gaps among its Hodge numbers. If otherwise, we have a Hodge structure with gaps.  
\newline
\noindent
\newline
So we see that each classifying space for variation of Hodge structures of even weight with given hodge numbers corresponds to an open $SO(p,q)$ orbit, for a certain $p$ and $q$. Conversely, each symmetric open $SO(p,q)$ orbit will correspond to a variation of Hodge structure for some weight and some hodge numbers. The same discussion applies to a dimension sequence $$(f_0,\dots, f_{m},f_{m},\dots, f_0),$$ where one obtains a one two one correspondence between $Sp(n,\R)$-open orbits with constant sign in each block and Hodge structures of odd weight.  This correspondence is described in more detail in \textit{section} \ref{combinatorics}.
\newline
\noindent
\newline
Choose a base point $z_0$ in $D$, i.e. a Hodge filtration corresponding to a Hodge decomposition $V_\C:=\oplus_{p+q=n}V^{p,q}$ such that $V^{q,p}=\overline{V}^{p,q}$, with $P$  the isotropy subgroup of $G$ at $z_0$. Then one can show that $P\cap G_0$, the isotropy subgroup of $G_0$ at $z_0$, is a compact subgroup of $G_0$ containing a compact maximal torus $T_0$ of $G_0$. Let $K_0$ denote the unique maximal compact subgroup of $G_0$ containing $P\cap G_0$, $\k_0$ its Lie algebra and $\theta$ the Cartan involution on $\g_0$ which is equal to the identity on $\k_0$ and $-$ the identity on $\mathfrak{m}_0$, the orthogonal complement of $k_0$ in $\g_0$ (with respect to the Killing form).   
\newline
\noindent
\newline
The base Hodge decomposition on $V_\C$ induces a Hodge decomposition of weight 0 on the Lie algebra $\g$ of $G$ defined by $\g:=\oplus_{f=-n}^{f=n}\g^{-f,f}, $ where for fixed $f$, $\g^{-f,f}$ is defined as the set of endomorphisms $g\in \g$ of $V_{\C}$ which act on the Hodge decomposition by sending $V^{p,q}$ to $V^{p-f,q+f}$, for all $p,q$ with $p+q=n$. As such the Lie algebra $\p$ of $P$ decomposes as $\p=\oplus_{f=-n}^0 \g^{-f,f}$, with Levi component $\mathfrak{s}:=\g^{0,0}$, unipotent radical $\mathfrak{u}^+:=\oplus _{f=-n}^{-1}\g^{-f,f}$ and opposite unipotent radical $\mathfrak{u}^{-}:=\oplus_{f=1}^n \g^{-f,f}$. The holomorphic tangent space of $D$ at $z_0$ can be naturally identified with $\g/\p$ which in turn can be naturally identified with $\mathfrak{u}^-$. Furthermore, if we denote by $\k$ and $\mathfrak{m}$ respectively, the complexification of $\k_0$ and $\mathfrak{m}_0$ respectively, and simply by $\theta$ the complex linear extension of $\theta$ to $\g$, then one can write $\g$ as a direct sum $\g=\k\oplus\mathfrak{m}$, where $\k=\oplus_{f \equiv 0 (\text{mod } 2)}\g^{-f,f}$ and $\mathfrak{m}=\oplus_{f\equiv 1(\text{mod }2)}\g^{-f,f}$. This decomposition is described in detail in \textit{section} \ref{combinatorics}.
\newline
\noindent
\newline
Denote by $\g^{-1,1}_B$, $\g^{-1,1}_C$, and $\g^{-1,1}_D$, the $\g^{-1,1}$-subspace for $\g=\socc$, $\g=\spc$, and $\g=\soc$ respectively. 

\subsection{Combinatorial structures associated to period domains}\label{combinatorics}
\textbf{Combinatorics of blocks in the case of $SO(2l,\mathbb{C})$.}
\newline
\noindent
\newline
Let $V$ be a finite dimensional $\R$- (or $\Q$-) vector space of dimension $2l$,  and consider a Hodge decomposition on $V_\C$ with $f_i:=h^{n-i,i}$, for all $0\le i \le m$, for a weight $n=2m$. By a choice of bases identify $V_\C$ with $\C^{2l}$ and fix $\{e_1,\dots,e_{2l}\}$ the ordered standard basis. 
\newline
\noindent
\newline
Let $a:=\sum_{i\equiv 0 (\text{mod } 2)} f_i$ and $b:=\sum_{i\equiv 1 (\text{mod } 2)} f_i$ and consider the polarisation described in section \ref{liestructure} so that the sequence $(f_0,\dots, f_{m-1},f_m,f_{m-1},\dots,f_0)$ up to nonzero elements, identifies the period domain associated to this set of Hodge numbers, with a flag domain of $SO(2a-f_m,2b)$ or $SO(2a, 2b-f_m)$, depending wether $m$ is even or odd, in the flag manifold of $SO(2l,\C)$. If $i\equiv 0 (\text{mod } 2)$, then the block $f_i$ has positive signature, and if $i\equiv 1 (\text{mod 2})$, then the block $f_i$ has negative signature. Let  
\begin{equation}\label{decompozitie1}
 \{1,\dots, a\}=
\begin{cases}
\mathcal{I}_0\cup \mathcal{I}_2\cup\dots \mathcal{I}_m^1,\, m \text{ even},\\
\mathcal{I}_0\cup \mathcal{I}_2\cup\dots \mathcal{I}_{m-1},\, m \text{ odd},
\end{cases}
\end{equation}
\begin{equation}\label{decompozitie2}
\{a+1,\dots, l\}=
\begin{cases}
\mathcal{I}_1\cup \mathcal{I}_3\cup\dots \mathcal{I}_{m-1},\, m \text{ even},\\
\mathcal{I}_1\cup \mathcal{I}_3\cup\dots \mathcal{I}_m^1,\, m \text{ odd},
\end{cases}
\end{equation}
be an order-preserving partition of the sets $\{1,\dots, a\}$ and $\{a+1,\dots, l\}$ such that the cardinality of each $\mathcal{I}_j$ equals $f_j$, for all $1\le j\le m-1$ and $|\mathcal{I}_m^1|=[f_m/2]$. Similarly, let  
\begin{equation}\label{decompo3}
 \{l+1,\dots,l+ a\}=
\begin{cases}
\mathcal{I}_{2m}\cup \mathcal{I}_{2m-2}\cup\dots \mathcal{I}_m^2,\, m \text{ even},\\
\mathcal{I}_{2m}\cup \mathcal{I}_{2m-2}\cup\dots \mathcal{I}_{m+1},\, m \text{ odd},
\end{cases}
\end{equation}
\begin{equation}\label{decompo4}
\{l+a+1,\dots, 2l\}=
\begin{cases}
\mathcal{I}_{2m-1}\cup \mathcal{I}_{2m-3}\cup\dots \mathcal{I}_{m+1},\, m \text{ even},\\
\mathcal{I}_{2m-1}\cup \mathcal{I}_{2m-3}\cup\dots \mathcal{I}_m^2,\, m \text{ odd},
\end{cases}
\end{equation}
be an order-preserving partition of the sets $\{l+1,\dots, l+a\}$ and $\{l+a+1,\dots, 2l\}$ such that the cardinality of each $\mathcal{I}_j$ equals $f_{n-j}$, for all $m+1\le j\le 2m$ and $|\mathcal{I}_m^2|=[f_m/2]$.
\newline
\noindent
\newline
Associated to the partitions (\ref{decompozitie1}), (\ref{decompozitie2}), (\ref{decompo3}), (\ref{decompo4}), consider the equivalent partition of the standard ordered basis in $\C^{2l}$ and choose the corresponding Hodge filtration $z_0\in D$ (equivalently Hodge decomposition $V_0$) as a base point in $D$. That is, $V^{n-k,k}$ is spanned by elements of the standard ordered basis indexed by $\mathcal{I}_k$, for all $k\ne m$. The middle member of the base Hodge decomposition, $V^{m,m}=\overline{V}^{m,m}$, is spanned by elements of the standard ordered basis indexed by $\mathcal{I}_m^1\cup{\mathcal{I}}_m^2$. 
\newline
\noindent
\newline
Each element $g\in \g$ acts on $V_\C$ via a linear transformation and after fixing the standard basis in $\C^{2l}$, the complex conjugation and polarisation as in section \ref{liestructure}, $g$ has a decomposition as a $2\times2$-block matrix 

$$g=\{\begin{pmatrix}
A & B \\
C & - A^{t}
\end{pmatrix}: \, B+B^{t}=0, C+ C^{t}=0, \,A,B,C\in\glc\}.$$
\noindent
The first purpose of this section is to further analyse the action of $\g$ on $V_\C$ via endomorphisms and further decompose the blocks $A$ (and implicitly $-A^t$), $B$ and $C$ of $g\in \g$, by looking at how $g$ acts on each subspace of the Hodge decomposition $V_0$. For this recall that an element $g$ in $\g$ acting on an element $e_j$ of the standard basis gives as a result the $j^{th}$ column of $g$, for all $j$. 
\newline
\noindent
\newline
Let $\mathcal{J}_1:=\{0,2,\dots,m,1,3,\dots,m-1\}$ for $m$ even and $\mathcal{J}_1:=\{0,2,\dots,m-1,1,3,\dots,m\}$ for $m$ odd respectively, denote the ordered set obtained by first considering the ordered even integers and then the ordered odd integers of the set $\{0,1,2,\dots,m\}$. Similarly, let $\mathcal{J}_2:=\{2m,\dots,m,2m-1,\dots,m+1\}$ for $m$ even and $\mathcal{J}_2:=\{2m,\dots,m+1,2m-1,\dots,m\}$ for $m$ odd respectively, denote the ordered set obtained by first considering the ordered even integers and then the ordered odd integers of the set $\{2m,2m-1,\dots,m\}$. 
\newline
Write $A$, $B$, $C$ and $(-A^t)$ respectively, as $(m+1)\times (m+1)$-block matrices, with rows and columns indexed by the ordered sets $\mathcal{J}_1\times \mathcal{J}_1$, $\mathcal{J}_1\times \mathcal{J}_2$, $\mathcal{J}_2\times \mathcal{J}_1$ and $\mathcal{J}_2\times \mathcal{J}_2$ respectively. 
\newline
The blocks $A_{ij}$, $B_{ij}$, $C_{ij}$ and $(-A_{ij}^t)$ respectively are of type $f_i\times f_j$, $f_i\times f_{2m-j}$, $f_{2m-i}\times f_{j}$ and $f_{2m-i}\times f_{2m-j}$ respectively, for all $i,j$ such that $i\ne m$. For $i=m$, these blocks are of type $f_m/2\times f_j$, $f_m/2\times f_{2m-j}$, $f_m/2\times f_j$, and $f_m/2\times f_{2m-j}$ respectively.
\newline
By symmetry $A_{ij}^t$ is identified with $-A_{n-j,n-i}^t$, $C_{ij}$ is identified with $C_{n-j,n-i}$ and $B_{ij}$ with $B_{n-j,n-i}$. For convenience of writing, denote the pair $(A,-A^t)$ simply by $A$ and the pairs $(A_{ij},-A_{n-j,n-i}^t)$, $(C_{ij},C_{n-j,n-i})$ and $(B_{ij},B_{n-j,n-i})$ respectively by $A_{ij}$, $C_{ij}$, and $B_{ij}$ respectively. One can think of each of  this blocks as elements of $\g$ by requiring any other entry in $\g$ outside the block to be zero. Furthermore, one can define the direct sum of two blocks, say for example $A_{ij}\oplus C_{fl}$, to be the element in $\g$ with non-zero entries only inside the blocks $A_{ij}$ and $C_{fl}$.
\newline
\noindent
\newline
We now look at how each such block (viewed as an element in $\g$) acts on the members of the Hodge decomposition. More precisely we have the following two cases: 
\begin{enumerate}
\item {The rows and columns of $A_{ij}$, $B_{ij}$, $C_{ij}$ and $(-A^t_{ij})$ respectively are indexed by $\mathcal{I}_i\times \mathcal{I}_j$ with the pair $(i,j)$ in the indexed set $\mathcal{J}_1\times \mathcal{J}_1$,$\mathcal{J}_1\times \mathcal{J}_2$, $\mathcal{J}_2\times \mathcal{J}_1$ and $\mathcal{J}_2\times \mathcal{J}_2$, respectively and this are precisely the blocks which map $V^{n-j,j}$ to $V^{n-i,i}$ for all $i,j$ not equal to $m$.}
\item{For $i=m$ or $j=m$, the blocks $A_{ij}$, $B_{ij}$, $C_{ij}$ and $(-A^t_{ij})$ respectively are one of the two blocks which map $V^{n-j,j}$ to $V^{n-i,i}$.}
\end{enumerate}
\noindent
Next we use the above decomposition to give a block decomposition of $\mathfrak{g}^{-p.p}$ for each $p\in \{-n,\dots,0,\dots,n\}$. By definition $\g^{0,0}$ is the set of endomorphisms in $\g$ which preserve each member of the Hodge decomposition, i.e. the Levi complement of the parabolic subalgebra $\p$. From the above description of the blocks action on the Hodge decomposition, this corresponds to $$\g^{0,0}=\bigoplus_{i\in \mathcal{J}_1}A_{ii}\bigoplus C_{m,m}\bigoplus B_{m,m},$$ i.e. to the diagonal blocks of $A$ together with the middle block of the diagonal in $C$ and $B$.
\newline
\noindent
\newline
Let $\h_{ii}$ denote the diagonal matrices inside the block $A_{ii}$, for all $0\le i \le m$. Then $$\h=\bigoplus_{i=0}^{m}\h_{ii},$$
and as a consequence $\h\subset \g^{0,0}$. 
\begin{prop}\label{orthogonaleven}
The $\g^{-p,p}$ piece can be decomposed as a direct sum of blocks, each indexed by pairs $(i,j)$ such that $i-j=p$, as follows:
\begin{itemize}
\item{
If $p>0$, then the subspace $\g^{-p,p}$, is the direct sum of $n-p+3$ blocks or $n-p+1$ blocks respectively, corresponding to the endomorphisms $V^{n-j,j}\mapsto V^{n-j-p,j+p}$, for all $0\le j \le n-p$, their total number depending respectively on wether the endomorphisms involve the $V^{m,m}$ member or not. If $j<m$ and $j+p<m$, then the blocks corresponding to these endomorphisms are $A_{j+p,j}$ and if $j>m$ and $j+p>m$, then the blocks corresponding to these endomorphisms are $(-A_{j+p,j}^t)$. If $j<m$ and $j+p>m$, then the blocks corresponding to these endomorphisms are $C_{j+p,j}$. If $j<m$ and $j+p=m$, then the endomorphism is given by the blocks $A_{m,j}\oplus C_{m,j}$ and if $j=m$ and $j+p>m$, then the endomorphism is given by the blocks $(-A_{j+p,j}^t)\oplus C_{j+p,j}$.Thus, in one formula we have:
$$\g^{-p,p}=\bigoplus_{0\le j< j+p\le m}A_{j+p,j}\bigoplus_{0\le j \le m, j+p\ge m}C_{j+p,j}  \bigoplus_{m\le j< j+p \le n}(-A_{j+p,j}^t).$$}
\item{
If $p<0$, then the subspace $\g^{-p,p}$, is the direct sum of $n+p+3$ blocks or $n+p+1$ blocks corresponding to the endomorphisms $V^{n-j+p,j-p}\mapsto V^{n-j,j}$, for all $0\le j \le n+p$, their total number depending respectively on wether the endomorphisms involve the $V^{m,m}$ member or not. 
Thus, in one formula, we have:
$$\g^{-p,p}=\bigoplus_{0\le j< j-p\le m}A_{j,j-p}\bigoplus_{0\le j\le m, j-p\ge m}B_{j,j-p}  \bigoplus_{m\le j < j-p \le n}(-A_{j,j-p}^t).$$}
\end{itemize}
\end{prop}
\noindent
\textbf{Example}. Let $l=10$, the group $G=SO(20,\C)$, the weight $n=8$ and thus $m=4$. It thus follows that the Hodge decomposition contains $9$ members, namely $$V_\C=V^{8,0}\oplus V^{7,1}\oplus V^{6,2}\oplus V^{5,3}\oplus V^{4,4}\oplus V^{3,5},\oplus V^{2,6}\oplus V^{1,7}\oplus V^{0,8}$$
and the signature sequence is given by $$(i^{8-0},i^{7-1},i^{6-2},i^{5-3},i^{4-4},i^{3-5},i^{2-6},i^{1-7},i^{0-8}),$$ which after computing the powers is $+-+-+-+-+$.
\noindent
Consider the following set of Hodge numbers: $h^{8,0}=2$, $h^{7,1}=3$, $h^{6,2}=2$, $h^{5,3}=1$, $h^{4,4}=4$. It thus follows that $G_0=SO(12,8)$ and the period domain is the $G_0$-open orbit parametrized by the signature sequence $+-+-+-+-+$ in the flag manifold $G/P$ parametrized by the dimension sequence $(2,3,2,1,4,1,2,3,2)$. The base point in $D$ is the Hodge filtration associated to the following Hodge decomposition: $V^{8,0}=<e_1,e_2> $, $V^{7,1}=<e_7,e_8,e_9>$, $V^{6,2}=<e_3,e_4>$, $V^{5,3}=<e_{10}>$, $V^{4,,4}=<e_5,e_6,e_{15},e_{16}>$, $V^{3,5}=<e_{20}>$,$V^{2,6}=<e_{13},e_{14}>$, $V^{1,7}=<e_{17},e_{18},e_{19}>$, $V^{0,8}=<e_{11},e_{12}>$.
Since $m=4$, it follows that $\mathcal{J}_1=\{0,2,4,1,3\}$, $\mathcal{J}_2=\{8,6,4,7,5\}$ and the blocks $A$, $B$, $C$ and $(-A^t)$ of any element $g\in \mathfrak{so}(20,\C)$ are further divided into $5\times 5$ block matrices.   

$$g=\begin{pmatrix}
A_{00} & A_{02} & A_{04} & A_{01} & A_{03} & B_{08} & B_{06} & B_{04} & B_{07} & B_{05} \\
A_{20} & A_{22} & A_{24} & A_{21} & A_{23} & B_{28} & B_{26} & B_{24} & B_{27} & B_{25} \\
A_{40} & A_{42} & A_{44} & A_{41} & A_{43} & B_{48} & B_{46} & B_{44} & B_{47} & B_{45} \\
A_{10} & A_{12} & A_{14} & A_{11} & A_{13} & B_{18} & B_{16} & B_{14} & B_{17} & B_{15} \\
A_{30} & A_{32} & A_{34} & A_{31} & A_{33} & B_{38} & B_{36} & B_{34} & B_{37} & B_{35} \\
C_{80} & C_{82} & C_{84} & C_{81} & C_{83} & -A_{88}^t & -A^t_{86} & -A^t_{84} & -A^t_{87} & -A^t_{85} \\
C_{60} & C_{62} & C_{64} & C_{61} & C_{63} & -A_{68}^t & -A^t_{66} & -A^t_{64} & -A^t_{67} & -A^t_{65} \\
C_{40} & C_{42} & C_{44} & C_{41} & C_{43} & -A_{48}^t & -A^t_{46} & -A^t_{44} & -A^t_{47} & -A^t_{45} \\
C_{70} & C_{72} & C_{74} & C_{71} & C_{73} & -A_{78}^t & -A^t_{76} & -A^t_{74} & -A^t_{77} & -A^t_{75} \\
C_{50} & C_{52} & C_{54} & C_{51} & C_{53} & -A_{58}^t & -A^t_{56} & -A^t_{54} & -A^t_{57} & -A^t_{55} 
\end{pmatrix}.$$
The $\k\oplus\mathfrak{m}$ decomposition and the decomposition of each $\g^{-p,p}$ into blocks is illustrated in the following equivalent matrix picture. The subalgebra $\k$ is isomorphic to the product $\mathfrak{so}(12,\C)\times \mathfrak{so}(8,\C)$, where the two factors are pictured with red and blue respectively. One observation is that choosing the standard basis in $\C^{2l}$ and the particular choice of indexing the blocks characterising the $\g^{-p,p}$-pieces gives a particularly nice matrix representation of the $\k\oplus\mathfrak{m}$ decomposition. Furthermore, the $\g^{-p,p}$-pieces can be visualised as diagonals in this matrix representations as the following picture shows. 
$$g=\begin{pmatrix}
\color{red}{\mathfrak{g}^{0,0}} & \color{red}{\mathfrak{g}^{2,-2}} & \color{red}{\mathfrak{g}^{4,-4}} & \mathfrak{g}^{1,-1} & \mathfrak{g}^{3,-3} & \color{red}{\mathfrak{g}^{8,-8}} & \color{red}{\mathfrak{g}^{6,-6}} & \color{red}{\mathfrak{g}^{4,-4}} & \mathfrak{g}^{7,-7} & \mathfrak{g}^{5,-5} \\
\color{red}{\mathfrak{g}^{-2,2}} & \color{red}{\mathfrak{g}^{0,0}} & \color{red}{\mathfrak{g}^{2,-2}} & \mathfrak{g}^{-1,1} & \mathfrak{g}^{1,-1} & \color{red}{\mathfrak{g}^{6,-6}} & \color{red}{\mathfrak{g}^{4,-4}} & \color{red}{\mathfrak{g}^{2,-2}} & \mathfrak{g}^{5,-5} & \mathfrak{g}^{3,-3} \\
\color{red}{\mathfrak{g}^{-4,4}} & \color{red}{\mathfrak{g}^{-2,2}} & \color{red}{\mathfrak{g}^{0,0}} & \mathfrak{g}^{-3,3} & \mathfrak{g}^{-1,1} & \color{red}{\mathfrak{g}^{4,-4}} & \color{red}{\mathfrak{g}^{2,-2}} & \color{red}{\mathfrak{g}^{0,0}} & \mathfrak{g}^{3,-3} & \mathfrak{g}^{1,-1} \\
\mathfrak{g}^{-1,1} & \mathfrak{g}^{1,-1} & \mathfrak{g}^{3,-3} & \color{blue}{\mathfrak{g}^{0,0}} & \color{blue}{\mathfrak{g}^{2,-2}} & \mathfrak{g}^{7,-7} & \mathfrak{g}^{5,-5} & \mathfrak{g}^{3,-3} & \color{blue}{\mathfrak{g}^{6,-6}} & \color{blue}{\mathfrak{g}^{4,-4}} \\
\mathfrak{g}^{-3,3} & \mathfrak{g}^{-1,1} & \mathfrak{g}^{1,-1} & \color{blue}{\mathfrak{g}^{-2,2}} & \color{blue}{\mathfrak{g}^{0,0}} & \mathfrak{g}^{5,-5} & \mathfrak{g}^{3,-3} & \mathfrak{g}^{1,-1} & \color{blue}{\mathfrak{g}^{-4,4}} & \color{blue}{\mathfrak{g}^{-2,2}} \\
\color{red}{\mathfrak{g}^{-8,8}} & \color{red}{\mathfrak{g}^{-6,6}} & \color{red}{\mathfrak{g}^{-4,4}} & \mathfrak{g}^{-7,7} & \mathfrak{g}^{-5,5} & \color{red}{\mathfrak{g}^{0,0}} & \color{red}{\mathfrak{g}^{-2,2}} & \color{red}{\mathfrak{g}^{-4,4}} & \mathfrak{g}^{-1,1} & \mathfrak{g}^{-3,3} \\
\color{red}{\mathfrak{g}^{-6,6}} & \color{red}{\mathfrak{g}^{-4,4}} & \color{red}{\mathfrak{g}^{-2,2}} & \mathfrak{g}^{-5,5} & \mathfrak{g}^{-3,3} & \color{red}{\mathfrak{g}^{2,-2}} & \color{red}{\mathfrak{g}^{0,0}} & \color{red}{\mathfrak{g}^{-2,2}} & \mathfrak{g}^{1,-1} & \mathfrak{g}^{-1,1} \\
\color{red}{\mathfrak{g}^{-4,4}} & \color{red}{\mathfrak{g}^{-2,2}} & \color{red}{\mathfrak{g}^{0,0}} & \mathfrak{g}^{-3,3} & \mathfrak{g}^{-1,1} & \color{red}{\mathfrak{g}^{4,-4}} & \color{red}{\mathfrak{g}^{2,-2}} & \color{red}{\mathfrak{g}^{0,0}} & \mathfrak{g}^{3,-3} & \mathfrak{g}^{1,-1} \\
\mathfrak{g}^{-7,7} & \mathfrak{g}^{-5,5} & \mathfrak{g}^{-3,3} & \color{blue}{\mathfrak{g}^{-6,6}} & \color{blue}{\mathfrak{g}^{-4,4}} & \mathfrak{g}^{1,-1} & \mathfrak{g}^{-1,1} & \mathfrak{g}^{-3,3} & \color{blue}{\mathfrak{g}^{0,0}} & \color{blue}{\mathfrak{g}^{-2,2}} \\
\mathfrak{g}^{-5,5} & \mathfrak{g}^{-3,3} & \mathfrak{g}^{-1,1} & \color{blue}{\mathfrak{g}^{-4,4}} & \color{blue}{\mathfrak{g}^{-2,2}} & \mathfrak{g}^{3,-3} & \mathfrak{g}^{1,-1} & \mathfrak{g}^{-1,1} & \color{blue}{\mathfrak{g}^{2,-2}} & \color{blue}{\mathfrak{g}^{0,0}} 
\end{pmatrix}.$$
In general, if $G_0=SO(2a-f_m,2b)$, then $K=SO(2a-f_m,\C)\times SO(2b,\C)$ and if $G_0=SO(2a,2b-f_m)$, then $K=SO(2a,\C)\times SO(2b-f_m/2,\C)$. Let $V_\C^{+}=\oplus_{p\equiv 0 (\text{mod } 2)} V^{n-p,p}$ and $V_\C^{-}=\oplus_{p\equiv 1 (\text{mod } 2)} V^{n-p,p}$. Then the minimal dimensional $K$-orbit in $Z$ is the set of isotropic flags in $Z$ which, roughly speaking, have optimal intersection with $V_{\C}^+$ and $V_{\C}^-$, i.e. flags $\mathcal{F}$ such that $\dim F^{n-p}\cap V_{\C}^+=  \oplus_{i\equiv 0 (\text{mod } 2)} f_i $ and $\dim F^{n-p}\cap V_{\C}^-=  \oplus_{i\equiv 1 (\text{mod } 2)} f_i $, for all $0\le p <m$.
\newline
\noindent
\newline
\textbf{Combinatorics of blocks in the case of $Sp(2l,\mathbb{C})$.}
\newline
\noindent
\newline
Let $V$ be a finite dimensional $\Q$-vector space of dimension $2l$,  and consider a Hodge decomposition on $V_\C$ with $f_i:=h^{n-i,i}$, for all $0\le i \le m$, for a weight $n=2m+1$. By a choice of bases identify $V_\C$ with $\C^{2l}$ and fix $\{e_1,\dots,e_{2l}\}$ the ordered standard basis. 
\newline
\noindent
\newline
Let $a:=\sum_{i\equiv 0 (\text{mod } 2)} f_i$ and $b:=\sum_{i\equiv 1 (\text{mod } 2)} f_i$ and consider the polarisation described in section \ref{symplectic} so that the sequence $(f_0,\dots, f_{m},f_m,\dots,f_0)$ up to nonzero elements, identifies the period domain associated to this set of Hodge numbers, with a flag domain of $Sp(a+b,b+a)$ with $a+b=l$ in the flag manifold of $Sp(2l,\C)$. If $i\equiv 0 (\text{mod } 2)$, then the block $f_i$ has negative signature, and if $i\equiv 1 (\text{mod 2})$, then the block $f_i$ has positive signature.   
Let 
\begin{equation}\label{decompozitie3}
 \{1,\dots, a\}=
\begin{cases}
\mathcal{I}_0\cup \mathcal{I}_2\cup\dots \mathcal{I}_m,\, m \text{ even},\\
\mathcal{I}_0\cup \mathcal{I}_2\cup\dots \mathcal{I}_{m-1},\, m \text{ odd},
\end{cases}
\end{equation}
\begin{equation}\label{decompozitie4}
\{a+1,\dots, l\}=
\begin{cases}
\mathcal{I}_1\cup \mathcal{I}_3\cup\dots \mathcal{I}_{m-1},\, m \text{ even},\\
\mathcal{I}_1\cup \mathcal{I}_3\cup\dots \mathcal{I}_m,\, m \text{ odd},
\end{cases}
\end{equation}
be an order-preserving partition of the sets $\{1,\dots, a\}$ and $\{a+1,\dots, l\}$ such that the cardinality of each $\mathcal{I}_j$ equals $f_j$, for all $1\le j\le m$. Similarly, let  
\begin{equation}\label{decompo7}
 \{l+1,\dots,l+ a\}=
\begin{cases}
\mathcal{I}_{2m+1}\cup \mathcal{I}_{2m-1}\cup\dots \mathcal{I}_{m+1},\, m \text{ even},\\
\mathcal{I}_{2m+1}\cup \mathcal{I}_{2m-1}\cup\dots \mathcal{I}_{m+2},\, m \text{ odd},
\end{cases}
\end{equation}
\begin{equation}\label{decompo8}
\{l+a+1,\dots, 2l\}=
\begin{cases}
\mathcal{I}_{2m}\cup \mathcal{I}_{2m-2}\cup\dots \mathcal{I}_{m+2},\, m \text{ even},\\
\mathcal{I}_{2m}\cup \mathcal{I}_{2m-2}\cup\dots \mathcal{I}_{m+1},\, m \text{ odd},
\end{cases}
\end{equation}
be an order-preserving partition of the sets $\{l+1,\dots, l+a\}$ and $\{l+a+1,\dots, 2l\}$ such that the cardinality of each $\mathcal{I}_j$ equals $f_{n-j}$, for all $m+1\le j\le 2m+1$.
\newline
\noindent
\newline
Associated to the partitions (\ref{decompozitie3}), (\ref{decompozitie4}), (\ref{decompo7}), (\ref{decompo8}) consider the equivalent partition of the standard ordered basis in $\C^{2l}$ and choose the corresponding Hodge filtration $z_0\in D$ (equivalently Hodge decomposition $V_0$) as a base point in $D$. That is, $V^{n-k,k}$ is spanned by elements of the standard ordered basis indexed by $\mathcal{I}_k$, for all $0\le k\le n$.
\newline
\noindent
\newline
Each element $g\in \g$ acts on $V_\C$ via a linear transformation and after fixing the standard basis in $\C^{2l}$, the complex conjugation and polarisation as in section \ref{liestructure}, $g$ has a decomposition as a $2\times2$-block matrix 

$$g=\{\begin{pmatrix}
A & B \\
C & A^{t}
\end{pmatrix}: \, B-B^{t}=0, C- C^{t}=0, \,A,B,C\in\glc\}.$$
\noindent
The first purpose of this section is to further analyse the action of $\g$ on $V_\C$ via endomorphisms and further decompose the blocks $A$ (and implicitly $A^t$), $B$ and $C$ of $g\in \g$, by looking at how $g$ acts on each subspace of the Hodge decomposition $V_0$. For this recall that an element $g$ in $\g$ acting on an element $e_j$ of the standard basis gives as a result the $j^{th}$ column of $g$, for all $j$. 
\newline
\noindent
\newline
Let $\mathcal{J}_1:=\{0,2,\dots,m,1,3,\dots,m-1\}$ for $m$ even and $\mathcal{J}_1:=\{0,2,\dots,m-1,1,3,\dots,m\}$ for $m$ odd respectively, denote the ordered set obtained by first considering the ordered even integers and then the ordered odd integers of the set $\{0,1,2,\dots,m\}$. Similarly, let $\mathcal{J}_2:=\{2m+1,2m-1,\dots,m+1,2m,2m-2\dots,m+2\}$ for $m$ even and $\mathcal{J}_2:=\{2m+1,2m-1,\dots,m+2,2m,2m-2\dots,m+1\}$ for $m$ odd respectively, denote the ordered set obtained by first considering the ordered odd integers and then the ordered even integers of the set $\{2m+1,2m,\dots,m+1\}$. 
\newline
Write $A$, $B$, $C$ and $(A^t)$ respectively, as $(m+1)\times (m+1)$-block matrices, with rows and columns indexed by the ordered sets $\mathcal{J}_1\times \mathcal{J}_1$, $\mathcal{J}_1\times \mathcal{J}_2$, $\mathcal{J}_2\times \mathcal{J}_1$ and $\mathcal{J}_2\times \mathcal{J}_2$ respectively. 
\newline
The blocks $A_{ij}$, $B_{ij}$, $C_{ij}$ and $(-A_{ij}^t)$ respectively are of type $f_i\times f_j$, $f_i\times f_{n-j}$, $f_{n-i}\times f_{j}$ and $f_{n-i}\times f_{n-j}$ respectively, for all $i,j$. 
\newline
By symmetry $(A_{ij})^t$ is identified with $A_{n-j,n-i}^t$, $C_{ij}$ is identified with $C_{n-j,n-i}$ and $B_{ij}$ with $B_{n-j,n-i}$. For convenience of writing, denote the pair $(A,A^t)$ simply by $A$ and the pairs $(A_{ij},A_{n-j,n-i}^t)$, $(C_{ij},C_{n-j,n-i})$ and $(B_{ij},B_{n-j,n-i})$ respectively by $A_{ij}$, $C_{ij}$, and $B_{ij}$ respectively. One can think of each of  this blocks as elements of $\g$ by requiring any other entry in $\g$ outside the block to be zero. Furthermore, one can define the direct sum of two blocks, say for example $A_{ij}\oplus C_{fl}$, to be the element in $\g$ with non-zero entries only inside the blocks $A_{ij}$ and $C_{fl}$.
\newline
\noindent
\newline
We now look at how each such block (viewed as an element in $\g$) acts on the memebrs of the Hodge decomposition. More precisely we have the following three cases:
\begin{itemize}
\item {The rows and columns of $A_{ij}$, $B_{ij}$, $C_{ij}$ and $(A^t_{ij})$ respectively are indexed by $\mathcal{I}_i\times \mathcal{I}_j$ with the pair $(i,j)$ in the indexed set $\mathcal{J}_1\times \mathcal{J}_1$,$\mathcal{J}_1\times \mathcal{J}_2$, $\mathcal{J}_2\times \mathcal{J}_1$ and $\mathcal{J}_2\times \mathcal{J}_2$, respectively and this are precisely the blocks which map $V^{n-j,j}$ to $V^{n-i,i}$ for all $i,j$.}
\end{itemize}
Next we use the above decomposition to give a block decomposition of $\mathfrak{g}^{-p.p}$ for each $p\in \{-n,\dots,0,\dots,n\}$. By definition $\g^{0,0}$ is the set of homomorphism in $\g$ which preserve each member of the Hodge decomposition, i.e. the Levi complement of the parabolic subalgebra $\p$. By our discussion this corresponds to $$\g^{0,0}=\bigoplus_{i\in \mathcal{J}_1}A_{ii},$$ i.e. to the diagonal blocks of $A$ together.
\newline
\noindent
\newline
Let $\h_{ii}$ denote the diagonal matrices inside the block $A_{ii}$, for all $0\le i \le m$. Then $$\h=\bigoplus_{i=0}^{m}\h_{ii},$$
and as a consequence $\h\subset \g^{0,0}$. 

\begin{prop}\label{symplectic}
The $\g^{-p,p}$ piece can be decomposed as a direct sum of blocks, each indexed by pairs $(i,j)$ such that $i-j$=p, as follows:
\begin{itemize}
\item{
If $p>0$, then the subspace $\g^{-p,p}$, is the direct sum of $n-p+1$ blocks corresponding to the endomorphisms $V^{n-j,j}\mapsto V^{n-j-p,j+p}$, for all $0\le j \le n-p$. If $j\le m$ and $j+p \le m$, then the blocks corresponding to this endomorphisms are $A_{j+p,j}$ and if $j \ge m$ and $j+p \ge m$, then the blocks corresponding to these endomorphisms are $(A^t_{j+p,j})$. If $j\le m,j+p \ge m$, then the blocks corresponding to this endomorphisms are $C_{j+p,j}$. Thus, in one formula we have:
$$\g^{-p,p}=\bigoplus_{0\le j<j+p\le m}A_{j+p,j}\bigoplus_{0\le j\le m, j+p> m}C_{j+p,j}  \bigoplus_{m<j<j+p \le n}(A_{j+p,j}^t).$$}

\item{
If $p<0$, then the subspace $\g^{-p,p}$, is the direct sum of $n+p+1$ blocks corresponding to the endomorphisms $V^{n-j+p,j-p}\mapsto V^{n-j,j}$, for all $0\le j \le n+p$. Thus, in one formula, we have:
$$\g^{-p,p}=\bigoplus_{0\le j<j-p\le m}A_{j,j-p}\bigoplus_{0\le j\le m, j-p> m}B_{j,j-p}  \bigoplus_{ m<j<j-p \le n}(A_{j,j-p}^t).$$}
\end{itemize}
\end{prop}
\noindent
\textbf{Example}. Let $l=7$, the group $G=Sp(14,\C)$, the weight $n=5$ and thus $m=2$. It thus follows that the Hodge decomposition contains $6$ members, namely $$V_\C=V^{5,0}\oplus V^{4,1}\oplus V^{3,2}\oplus V^{2,3}\oplus V^{1,4}\oplus V^{0,5}$$ and the signature sequence is given by $$(i^{5-0},i^{4-1},i^{3-2},i^{2-3},i^{1-4},i^{0-5}),$$ which after computing the powers is $-+-+-+$.
\noindent
Consider the following set of Hodge numbers: $h^{5,0}=2$, $h^{4,1}=3$, $h^{3,2}=2$. It thus follows that $G_0=Sp(4+3,4+3)$ and the period domain is the $G_0$-open orbit parametrized by the signature sequence $-+-+-+$ in the flag manifold $G/P$ parametrized by the dimension sequence $(2,3,2,2,3,2)$. The base point in $D$ is the Hodge filtration associated to the following Hodge decomposition: $V^{5,0}=<e_1,e_2> $, $V^{4,1}=<e_5,e_6,e_7>$, $V^{3,2}=<e_3,e_4>$, $V^{2,3}=<e_{10},e_{11}>$,$V^{1,4}=<e_{12},e_{13},e_{14}>$, $V^{0,5}=<e_{8},e_{9}>$.
Since $m=2$, it follows that $\mathcal{J}_1=\{0,2,1\}$, $\mathcal{J}_2=\{5,3,4\}$ and the blocks $A$, $B$, $C$ and $(A^t)$ of any element $g\in \mathfrak{sp}(14,\C)$ are further divided into $3\times 3$ block matrices.   

$$g=\begin{pmatrix}
A_{00} & A_{02} & A_{01} & B_{05} & B_{03} & B_{04} \\
A_{20} & A_{22} & A_{21} & B_{25} & B_{23} & B_{24}  \\
A_{10} & A_{12} & A_{11} & B_{15} & B_{13} & B_{14}  \\
C_{50} & C_{52} & C_{51} & A_{55}^t & A^t_{53} & A^t_{54} \\
C_{30} & C_{32} & C_{31} & A_{35}^t & A^t_{33} & A^t_{34}  \\
C_{40} & C_{42} & C_{41} & A_{45}^t & A^t_{43} & A^t_{44}  
\end{pmatrix}.$$
The $\k\oplus\mathfrak{m}$ decomposition and the decomposition of each $\g^{-p,p}$ into blocks is illustrated in the following equivalent matrix picture. The subalgebra $\k$ is isomorphic to the (anti-)diagonally embedded $\mathfrak{sl}(7,\C)$ pictured in red. One observation is that choosing the standard basis in $\C^{2l}$ and the particular choice of indexing the blocks characterising the $\g^{-p,p}$-pieces gives a particularly nice matrix representation of the $\k\oplus\mathfrak{g}$ decomposition. Furthermore, the $\g^{-p,p}$-pieces can be visualised as diagonals in this matrix representations as the following picture shows. 
$$g=\begin{pmatrix}
\color{red}{\mathfrak{g}^{0,0}} & \color{red}{\mathfrak{g}^{2,-2}} & \mathfrak{g}^{1,-1} & {\mathfrak{g}^{5,-5}} & {\mathfrak{g}^{3,-3}} & \color{red}{\mathfrak{g}^{4,-4}} \\
\color{red}{\mathfrak{g}^{-2,2}} & \color{red}{\mathfrak{g}^{0,0}} & \mathfrak{g}^{-1,1} & {\mathfrak{g}^{3,-3}} & {\mathfrak{g}^{1,-1}} & \color{red}{\mathfrak{g}^{2,-2}} \\
\mathfrak{g}^{-1,1} & \mathfrak{g}^{1,-1} & \color{red}{\mathfrak{g}^{0,0}} & \color{red}{\mathfrak{g}^{4,-4}} & \color{red}{\mathfrak{g}^{2,-2}} & \mathfrak{g}^{-3,3} \\
{\mathfrak{g}^{-5,5}} & {\mathfrak{g}^{-3,3}} & \color{red}{\mathfrak{g}^{-4,4}}  & \color{red}{\mathfrak{g}^{0,0}} & \color{red}{\mathfrak{g}^{-2,2}} & \mathfrak{g}^{-1,1} \\
{\mathfrak{g}^{-3,3}} & {\mathfrak{g}^{-1,1}} & \color{red}{\mathfrak{g}^{-2,2}}  & \color{red}{\mathfrak{g}^{-2,2}} & \color{red}{\mathfrak{g}^{0,0}} & \mathfrak{g}^{1,-1} \\
 \color{red}{\mathfrak{g}^{-4,4}} & \color{red}{\mathfrak{g}^{-2,2}} & \mathfrak{g}^{-3,3} & \mathfrak{g}^{1,-1} & \mathfrak{g}^{-1,1} & \color{red}{\mathfrak{g}^{0,0}}  
\end{pmatrix}.$$
In general, $K=Sl(l,\C)$. Let $V_\C^{-}=\oplus_{p\equiv 0 (\text{mod } 2)} V^{n-p,p}$ and $V_\C^{+}=\oplus_{p\equiv 1 (\text{mod } 2)} V^{n-p,p}$. Then the minimal dimensional $K$-orbit in $Z$ is the set of isotropic flags in $Z$ which, roughly speaking, have optimal intersection with $V_{\C}^+$ and $V_{\C}^-$, i.e. flags $\mathcal{F}$ such that $\dim F^{n-p}\cap V_{\C}^+=  \oplus_{i\equiv 0 (\text{mod } 2)} f_i $ and $\dim F^{n-p}\cap V_{\C}^-=  \oplus_{i\equiv 1 (\text{mod } 2)} f_i $, for all $0\le p \le m$.
\newline
\noindent
\newline
\textbf{Combinatorics of blocks in the case of $SO(2l+1,\mathbb{C})$.}
\newline
\noindent
\newline
Let $V$ be a finite dimensional $\R-$ (or $\Q$)-vector space of dimension $2l+1$,  and consider a Hodge decomposition on $V_\C$ with $f_i:=h^{n-i,i}$, for all $0\le i \le m$, for a weight $n=2m$. By a choice of bases identify $V_\C$ with $\C^{2l+1}$ and fix $\{e_1,\dots,e_{2l+1}\}$ the ordered standard basis. 
\newline
\noindent
\newline
Let $a:=\sum_{i\equiv 0 (\text{mod } 2)} f_i$ and $b:=\sum_{i\equiv 1 (\text{mod } 2)} f_i$ and consider the polarisation described in section \ref{liestructure} so that the sequence $(f_0,\dots, f_{m-1},f_m,f_{m-1},\dots,f_0)$ up to nonzero elements, identifies the period domain associated to this set of Hodge numbers, with a flag domain of $SO(2a-f_m,2b)$ or $SO(2a, 2b-f_m)$, depending wether $m$ is even or odd, in the flag manifold of $SO(2l+1,\C)$. If $i\equiv 0 (\text{mod } 2)$, then the block $f_i$ has positive signature, and if $i\equiv 1 (\text{mod 2})$, then the block $f_i$ has negative signature. Let  
\begin{equation}\label{decompo9}
 \{1,\dots, a\}=
\begin{cases}
\mathcal{I}_0\cup \mathcal{I}_2\cup\dots \mathcal{I}_m^1,\, m \text{ even},\\
\mathcal{I}_0\cup \mathcal{I}_2\cup\dots \mathcal{I}_{m-1},\, m \text{ odd},
\end{cases}
\end{equation}
\begin{equation}\label{decompo10}
\{a+1,\dots, l\}=
\begin{cases}
\mathcal{I}_1\cup \mathcal{I}_3\cup\dots \mathcal{I}_{m-1},\, m \text{ even},\\
\mathcal{I}_1\cup \mathcal{I}_3\cup\dots \mathcal{I}_m^1,\, m \text{ odd},
\end{cases}
\end{equation}
be an order-preserving partition of the sets $\{1,\dots, a\}$ and $\{a+1,\dots, l\}$ such that the cardinality of each $\mathcal{I}_j$ equals $f_j$, for all $1\le j\le m-1$ and $|\mathcal{I}_m^1|=[f_m/2]$. Similarly, let  
\begin{equation}\label{decompo11}
 \{l+2,\dots,l+ a+1\}=
\begin{cases}
\mathcal{I}_{2m}\cup \mathcal{I}_{2m-2}\cup\dots \mathcal{I}_m^2,\, m \text{ even},\\
\mathcal{I}_{2m}\cup \mathcal{I}_{2m-2}\cup\dots \mathcal{I}_{m+1},\, m \text{ odd},
\end{cases}
\end{equation}
\begin{equation}\label{decompo12}
\{l+a+2,\dots, 2l+1\}=
\begin{cases}
\mathcal{I}_{2m-1}\cup \mathcal{I}_{2m-3}\cup\dots \mathcal{I}_{m+1},\, m \text{ even},\\
\mathcal{I}_{2m-1}\cup \mathcal{I}_{2m-3}\cup\dots \mathcal{I}_m^2,\, m \text{ odd},
\end{cases}
\end{equation}
be an order-preserving partition of the sets $\{l+1,\dots, l+a+1\}$ and $\{l+a+2,\dots, 2l+1\}$ such that the cardinality of each $\mathcal{I}_j$ equals $f_{n-j}$, for all $m+1\le j\le 2m$ and $|\mathcal{I}_m^2|=[f_m/2]$. 
\newline
\noindent
\newline
Associated to the partitions (\ref{decompo9}), (\ref{decompo10}), (\ref{decompo11}), (\ref{decompo12}), consider the equivalent partition of the standard ordered basis in $\C^{2l}$ and choose the corresponding Hodge filtration $z_0\in D$ (equivalently Hodge decomposition $V_0$) as a base point in $D$. That is, $V^{n-k,k}$ is spanned by elements of the standard ordered basis indexed by $\mathcal{I}_k$, for all $k\ne m$. The middle member of the base Hodge decomposition, $V^{m,m}=\overline{V}^{m,m}$, is spanned by elements of the standard ordered basis indexed by $\mathcal{I}_m^1\cup{\mathcal{I}}_m^2\cup\mathcal{I}_m^3$, where $\mathcal{I}_{3}^m$ is the one element set containing $l+1$. 
\newline
\noindent
\newline
Each element $g\in \g$ acts on $V_\C$ via a linear transformation and after fixing the standard basis in $\C^{2l+1}$, the complex conjugation and polarisation as in section \ref{liestructure}, $g$ has a decomposition as a $3\times3$-block matrix 
$$g=\{\begin{pmatrix}
A & x & B \\
y^t & 0 & -x^t \\
C & -y  & - A^{t}
\end{pmatrix}: \, B+B^{t}=0, C+ C^{t}=0, \,A,B,C\in \glc\, x, y\in \C^l\}.$$
\noindent
The first purpose of this section is to further analyse the action of $\g$ on $V_\C$ via endomorphisms and further decompose the blocks $A$, $B$, $C$ , $(-A^t)$, $(x,-x^t)$ and $(y,-y^t)$ of $g\in \g$, by looking at how $g$ acts on each subspace of the Hodge decomposition $V_0$. For this recall that an element $g$ in $\g$ acting on an element $e_j$ of the standard basis gives as a result the $j^{th}$ column of $g$, for all $j$. 
\newline
\noindent
\newline
Let $\mathcal{J}_1:=\{0,2,\dots,m,1,3,\dots,m-1\}$ for $m$ even and $\mathcal{J}_1:=\{0,2,\dots,m-1,1,3,\dots,m\}$ for $m$ odd respectively, denote the ordered set obtained by first considering the ordered even integers and then the ordered odd integers of the set $\{0,1,2,\dots,m\}$. Similarly, let $\mathcal{J}_2:=\{2m,\dots,m,2m-1,\dots,m+1\}$ for $m$ even and $\mathcal{J}_2:=\{2m,\dots,m+1,2m-1,\dots,m\}$ for $m$ odd respectively, denote the ordered set obtained by first considering the ordered even integers and then the ordered odd integers of the set $\{2m,2m-1,\dots,m\}$. 
\newline
Write $A$, $B$, $C$ and $(-A^t)$ respectively, as $(m+1)\times (m+1)$-block matrices, with rows and columns indexed by the ordered sets $\mathcal{J}_1\times \mathcal{J}_1$, $\mathcal{J}_1\times \mathcal{J}_2$, $\mathcal{J}_2\times \mathcal{J}_1$ and $\mathcal{J}_2\times \mathcal{J}_2$ respectively. 
\newline
The blocks $A_{ij}$, $B_{ij}$, $C_{ij}$ and $(-A_{ij}^t)$ respectively are of type $f_i\times f_j$, $f_i\times f_{2m-j}$, $f_{2m-i}\times f_{j}$ and $f_{2m-i}\times f_{2m-j}$ respectively, for all $i,j$ such that $i\ne m$. For $i=m$, these blocks are of type $f_m/2\times f_j$, $f_m/2\times f_{2m-j}$, $f_m/2\times f_j$, and $f_m/2\times f_{2m-j}$ respectively.
\newline
By symmetry $A_{ij}^t$ is identified with $-A_{n-j,n-i}^t$, $C_{ij}$ is identified with $C_{n-j,n-i}$ and $B_{ij}$ with $B_{n-j,n-i}$. For convenience of writing, denote the pair $(A,-A^t)$ simply by $A$ and the pairs $(A_{ij},-A_{n-j,n-i}^t)$, $(C_{ij},C_{n-j,n-i})$ and $(B_{ij},B_{n-j,n-i})$ respectively by $A_{ij}$, $C_{ij}$, and $B_{ij}$ respectively. One can think of each of  this blocks as elements of $\g$ by requiring any other entry in $\g$ outside the block to be zero. Furthermore, one can define the direct sum of two blocks, say for example $A_{ij}\oplus C_{fl}$, to be the element in $\g$ with non-zero entries only inside the blocks $A_{ij}$ and $C_{fl}$.
\newline
\noindent
\newline
Moreover, write $x$, $(-x^t)$, $y$ and $(-y^t)$ respectively as $(m+1)\times 1$, $1\times (m+1)$, $1\times (m+1)$ and $(m+1)\times 1$ blocks respectively, indexed by $\mathcal{J}_1\times \{2\}$, $\{2\}\times \mathcal{J}_2$, $\{2\}\times \mathcal{J}_1$ and $\mathcal{J}_2\times \{2\}$ respectively. 
\newline
\noindent
\newline
We now look at how each such block (viewed as an element in $\g$) acts on the members of the Hodge decomposition. More precisely we have the following three cases: 
\begin{enumerate}
\item {The rows and columns of $A_{ij}$, $B_{ij}$, $C_{ij}$ and $(-A^t_{ij})$ respectively are indexed by $\mathcal{I}_i\times \mathcal{I}_j$ with the pair $(i,j)$ in the indexed set $\mathcal{J}_1\times \mathcal{J}_1$,$\mathcal{J}_1\times \mathcal{J}_2$, $\mathcal{J}_2\times \mathcal{J}_1$ and $\mathcal{J}_2\times \mathcal{J}_2$, respectively and this are precisely the blocks which map $V^{n-j,j}$ to $V^{n-i,i}$ for all $i,j$ not equal to $m$.}
\item{For $i=m$ or $j=m$, the blocks $A_{ij}$, $B_{ij}$, $C_{ij}$ and $(-A^t_{ij})$ respectively are one of the three blocks which map $V^{n-j,j}$ to $V^{n-i,i}$.}
\item{The block $x_{i,m}$ corresponds to the endomorphism which sends $e_{l+1}$ to $V^{n-i,i}$ for all $i \in \mathcal{J}_1$ and the block $y_{m,i}$ corresponds to the endomorphism which sends $V^{n-i,i}$ to $e_{l+1}$ for all $i\in \mathcal{J}_1$. Similarly, 
the block $(-x_{m,i}^t)$ corresponds to the endomorphism which sends $V^{n-i,i}$ to $e_{l+1}$ for all $i \in \mathcal{J}_2$ and the block $(-y_{i,m})^t$ corresponds to the endomorphism which sends $e_{l+1}$ to $V^{n-i,i}$ for all $i\in \mathcal{J}_2$.}
\end{enumerate}
\noindent
Next we use the above decomposition to give a block decomposition of $\mathfrak{g}^{-p.p}$ for each $p\in \{-n,\dots,0,\dots,n\}$. By definition $\g^{0,0}$ is the set of homomorphism in $\g$ which preserve each member of the Hodge decomposition, i.e. the Levi complement of the parabolic subalgebra $\p$. By our discussion this corresponds to $$\g^{0,0}=\bigoplus_{i\in \mathcal{J}_1}A_{i,i}\bigoplus C_{m,m}\bigoplus B_{m,m}\bigoplus y_{m,m}\bigoplus x_{m,m},$$ i.e. the diagonal blocks of $A$ together with the middle block of the diagonal in $C$ and $B$.
\newline
\noindent
\newline
Let $\h_{ii}$ denote the diagonal matrices inside the block $A_{ii}$, for all $0\le i \le m$. Then $$\h=\bigoplus_{i=0}^{m}\h_{ii},$$
and as a consequence $\h\subset \g^{0,0}$. 
\begin{prop}\label{orthogonalodd}
The $\g^{-p,p}$ piece can be decomposed as a direct sum of blocks, each indexed by pairs $(i,j)$ such that $i-j=p$, as follows:
\begin{itemize}
\item{
If $p>0$, then the subspace $\g^{-p,p}$, is the direct sum of $n-p+5$ blocks or $n-p+1$ blocks respectively, corresponding to the endomorphisms $V^{n-j,j}\mapsto V^{n-j-p,j+p}$, for all $0\le j \le n-p$, their total number depending respectively on wether the endomorphisms involve the $V^{m,m}$ member or not. If $j<m$ and $j+p<m$, then the blocks corresponding to these endomorphisms are $A_{j+p,j}$ and if $j>m$ and $j+p>m$, then the blocks corresponding to these endomorphisms are $(-A_{j+p,j}^t)$. If $j<m$ and $j+p>m$, then the blocks corresponding to these endomorphisms are $C_{j+p,j}$. If $j<m$ and $j+p=m$, then the endomorphism is given by the blocks $A_{m,j}\oplus C_{m,j}\oplus y_{m,j}$ and if $j=m$ and $j+p>m$, then the endomorphism is given by the blocks $(-A_{j+p,m}^t)\oplus C_{j+p,m} \oplus {(-y^t_{j+p,m})}$.Thus, in one formula, ommiting the symmetric blocks, we have:
$$\g^{-p,p}=
\bigoplus_{0\le j <j+p\le m}A_{j+p,j}\bigoplus_{0\le j\le m, j+p\ge m}C_{j+p,j} 
 \bigoplus_{m \le j <j+p \le n}(-A_{j+p,j}^t)$$$$\bigoplus_{j+p=m} {y}_{j+p,j}\bigoplus_{j=m}{y}^t_{j+p,j}.$$
 }

\item{
If $p<0$, then the subspace $\g^{-p,p}$, is the direct sum of $n+p+5$ blocks or $n+p+1$ blocks corresponding to the endomorphisms $V^{n-j+p,j-p}\mapsto V^{n-j,j}$, for all $0\le j \le n+p$, their total number depending respectively on wether the endomorphisms involve the $V^{m,m}$ member or not. 
Thus, in one formula, we have:
$$\g^{-p,p}=\bigoplus_{0\le j<j-p\le m}A_{j,j-p}\bigoplus_{0\le j\le m, j-p\ge m}B_{j,j-p}  \bigoplus_{m\le j<j-p \ge n}(-A_{j,j-p}^t)$$$$\bigoplus_{j-p=m} {x}_{j,j-p}\bigoplus_{j=m}{x}^t_{j,j-p}.$$}
\end{itemize}
\end{prop}
\noindent

\noindent
\textbf{Example}. Let $l=5$, the group $G=SO(11,\C)$, the weight $n=4$ and thus $m=2$. It thus follows that the Hodge decomposition contains $5$ members, namely $$V_\C=V^{4,0}\oplus V^{3,1}\oplus V^{2,2}\oplus V^{1,3}\oplus V^{0,4}$$
and the signature sequence is given by $(i^{4-0},i^{3-1},i^{2-2},i^{1-3},i^{0-4}),$ which after computing the powers is $+-+-+$.
\noindent
Consider the following set of Hodge numbers: $h^{4,0}=2$, $h^{3,1}=2$, $h^{2,2}=3$. It thus follows that $G_0=SO(7,4)$ and the period domain is the $G_0$-open orbit parametrized by the signature sequence $+-+-+$ in the flag manifold $G/P$ parametrized by the dimension sequence $(2,2,3,2,2)$. The base point in $D$ is the Hodge filtration associated to the following Hodge decomposition: $V^{4,0}=<e_1,e_2> $, $V^{3,1}=<e_4,e_5>$, $V^{2,2}=<e_3,e_9,e_6>$, $V^{1,3}=<e_{10},e_{11}>$, $V^{0,4}=<e_7,e_8>$. Since $m=2$, it follows that $\mathcal{J}_1=\{0,2,1\}$, $\mathcal{J}_2=\{4,2,3\}$ and the blocks $A$, $B$, $C$ and $(-A^t)$ of any element $g\in \mathfrak{so}(11,\C)$ are further divided into $3\times 3$, $3\times 1$ and $1\times 3$ block matrices, respectively.   

$$g=\begin{pmatrix}
A_{00} & A_{02} & A_{01} & x_{02} & B_{04} & B_{02} & B_{03} \\
A_{20} & A_{22} & A_{21} & x_{22} & B_{24} & B_{22} & B_{23}  \\
A_{10} & A_{12} & A_{11} & x_{12} & B_{14} & B_{12} & B_{13}  \\
y_{20}^t & y_{22}^t & y_{21}^t & 0 & -x_{24}^{t} & -x_{22}^t & -x_{23}^t  \\
C_{40} & C_{42} & C_{41} & -y_{42} & -A^t_{44} & -A^t_{42} & -A^t_{43} \\
C_{20} & C_{22} & C_{21} & -y_{22} & -A^t_{24} & -A^t_{22} & -A^t_{23} \\
C_{30} & C_{32} & C_{31} & -y_{32} & -A^t_{34} & -A^t_{32} & -A^t_{33} 
\end{pmatrix}.$$
The $\k\oplus\mathfrak{m}$ decomposition and the decomposition of each $\g^{-p,p}$ into blocks is illustrated in the following equivalent matrix picture. The subalgebra $\k$ is isomorphic to the product $\mathfrak{so}(7,\C)\times \mathfrak{so}(4,\C)$, where the two factors are pictured with red and blue respectively. One observation is that choosing the standard basis in $\C^{2l}$ and the particular choice of indexing the blocks characterising the $\g^{-p,p}$-pieces gives a particularly nice matrix representation of the $\k\oplus \m$ decomposition. Furthermore, the $\g^{-p,p}$-pieces can be visualised as diagonals in this matrix representations as the following picture shows. 
$$g=\begin{pmatrix}
\color{red}{\mathfrak{g}^{0,0}} & \color{red}{\mathfrak{g}^{2,-2}} & \mathfrak{g}^{1,-1} & \color{red}{\mathfrak{g}^{2,-2}} & \color{red}{\mathfrak{g}^{4,-4}}  & \color{red}{\mathfrak{g}^{2,-2}} & \mathfrak{g}^{3,-3} \\
\color{red}{\mathfrak{g}^{-2,2}} & \color{red}{\mathfrak{g}^{0,0}} & \mathfrak{g}^{-1,1}  & \color{red}{\mathfrak{g}^{0,0}} & \color{red}{\mathfrak{g}^{2,-2}} & \color{red}{\mathfrak{g}^{0,0}} & \mathfrak{g}^{1,-1} \\
\mathfrak{g}^{-1,1} & \mathfrak{g}^{1,-1} & \color{blue}{\mathfrak{g}^{0,0}} & \mathfrak{g}^{1,-1} & \mathfrak{g}^{3,-3} & \mathfrak{g}^{1,-1} & \color{blue}{\mathfrak{g}^{2,-2}}  \\
\color{red}{\mathfrak{g}^{-2,2}} & \color{red}{\mathfrak{g}^{0,0}} & \mathfrak{g}^{-1,1} & & \color{red}{\mathfrak{g}^{2,-2}} & \color{red}{\mathfrak{g}^{0,0}} & \mathfrak{g}^{1,-1} \\
\color{red}{\mathfrak{g}^{-4,4}} & \color{red}{\mathfrak{g}^{-2,2}}  & \mathfrak{g}^{-3,3} & \color{red}{\mathfrak{g}^{-2,2}} & \color{red}{\mathfrak{g}^{0,0}} & \color{red}{\mathfrak{g}^{-2,2}}  & \mathfrak{g}^{-1,1} \\
\color{red}{\mathfrak{g}^{-2,2}} & \color{red}{\mathfrak{g}^{0,0}}  & \mathfrak{g}^{-1,1} & \color{red}{\mathfrak{g}^{0,0}} & \color{red}{\mathfrak{g}^{2,-2}} & \color{red}{\mathfrak{g}^{0,0}}  & \mathfrak{g}^{1,-1} \\
\mathfrak{g}^{-3,3} & \mathfrak{g}^{-1,1} & \color{blue}{\mathfrak{g}^{-2,2}} & \mathfrak{g}^{-1,1} & \mathfrak{g}^{1,-1} & \mathfrak{g}^{-1,1} &\color{blue}{\mathfrak{g}^{0,0}} 
\end{pmatrix}.$$
\noindent
In general, if $G_0=SO(2a-f_m,2b)$, then $K=SO(2a-f_m,\C)\times SO(2b,\C)$ and if $G_0=SO(2a,2b-f_m)$, then $K=SO(2a,\C)\times SO(2b-f_m,\C)$. Let $V_\C^{+}=\oplus_{p\equiv 0 (\text{mod } 2)} V^{n-p,p}$ and $V_\C^{-}=\oplus_{p\equiv 1 (\text{mod } 2)} V^{n-p,p}$. Then the minimal dimensional $K$-orbit in $Z$ is the set of isotropic flags in $Z$ which, roughly speaking, have optimal intersection with $V_{\C}^+$ and $V_{\C}^-$, i.e. flags $\mathcal{F}$ such that $\dim F^{n-p}\cap V_{\C}^+=  \oplus_{i\equiv 0 (\text{mod } 2)} f_i $ and $\dim F^{n-p}\cap V_{\C}^-=  \oplus_{i\equiv 1 (\text{mod } 2)} f_i $, for all $0\le p <m$.

\subsection{Main results}\label{results}
\noindent
With the notation from the previous section, the following definitions hold for all the groups. 
\begin{defi}\label{Hodgetriple}
Call any pair $D_{ij}:=(A_{ii},A_{jj})$ for each $0 \le i<j  \le m$, a \textbf{semisimple pair} and any triple $$H_{ij}:=[A_{ij},D_{ij},A_{ji}]$$ a \textbf{ standard Hodge triple}. Equivalently, call any pair $D_{n-i,j}:=(A_{n-i,n-i},A_{jj})$ for each $m \le i  \le n$, $0\le j \le m$, $n-i\ne j$, a \textbf{semisimple pair} and any triple $$H_{ij}^c:=[C_{ij},D_{n-i,j},B_{n-i,n-j}]$$ a \textbf{ standard Hodge triple}
\end{defi}
\begin{defi}
For $m \le i \le n$ and $n-i=j$, call any of the triples $$H_{ij}^c=[C_{ij}, A_{n-i,j}, B_{n-i,n-j}]$$ for the even dimensional case and
$$H_{ij}^c=[(C_{ij},y_{mj}), A_{n-i,j}, (x_{jm}B_{n-i,n-j})]$$ for the odd dimensional case respectively, a \textbf{non-standard Hodge triple} of type $(C)$ or of type $(B)$ respectively.
\end{defi}

\noindent
In detail, a Hodge triple consists of the following data: 
\begin{itemize}
\item{ a semisimple pair $(A_{ii},A_{jj})\subset \g^{0,0}$ for $i < j$ and simply $A_{ii}$ for $i=j$. Implicitly, this corresponds also to a choice $$\h_{ij}:=\h_{ii}\oplus\h_{jj}$$ of subalgebra of the Cartan subalgebra $\h$, with $\h_{ii}$ and $\h_{jj}$ respectively, denoting the diagonal matrices of the blocks $A_{ii}$ and $A_{jj}$, respectively.}
\item{a unipotent pair $(A_{ij}, A_{ji})\in \g^{-p,p}\oplus \g^{p,-p}$, for $i\ne j$ and $p=i-j$ or a pair $(C_{ij},B_{n-i,n-j}) \in \g^{-p,p}\oplus \g^{p,-p}$, for $n-i\ne j$, $p=i-j$.}
\item{a unipotent pair $(C_{ij}, B_{n-i,n-j})\in \g^{-p,p}\oplus \g^{p,-p}$, for $n-i= j$ and $p=i-j$ or $((C_{ij}, y_{mj}),(x_{jm}, B_{n-i,n-j}))$. Observe that in the odd dimensional case when $j=n-i$, in general $C_{ij}$ and $y_{mj}$ do not belong to the same $\g^{-p,p}$ piece.}
\item{a subset of roots $$\theta_{ij}:=\{-\varepsilon_{i_1}+\varepsilon_{i_2}:\,i_1\in \mathcal{I}_{j}, \, i_2\in \mathcal{I}_{i}\}
\,\cup\,\theta_{ji}:=\{-\varepsilon_{i_1}+\varepsilon_{i_2}:\,i_1\in \mathcal{I}_{i}, \, i_2\in \mathcal{I}_{j}\}$$} for the pair $(A_{ij}, A_{ji})$ 
\item{a subset of roots $$\theta_{ij}^c:=\{-\varepsilon_{i_1}-\varepsilon_{i_2}:\,i_1\in \mathcal{I}_{j}, \, i_2\in \mathcal{I}_{n-i}\}
\,\cup\,\theta_{ji}^b:=\{\varepsilon_{i_1}+\varepsilon_{i_2}:\,i_1\in \mathcal{I}_{j}, \, i_2\in \mathcal{I}_{n-i}\}$$} for the pair $(C_{ij}, B_{n-i,n-j})$.
\item{a subset of roots $$\theta_{ij}^c:=\{-\varepsilon_{i_1}-\varepsilon_{i_2}:\,i_1\ne i_2\in \mathcal{I}_{n-i}\}
\,\cup\,\theta_{ji}^b:=\{\varepsilon_{i_2}+\varepsilon_{i_1}:\,i_1\ne i_2\in \mathcal{I}_{n-i}\}$$} 
for the pair $(C_{ij}, B_{n-i,n-j})$ with $n-i=j$ together with the roots $\theta^{d}_{ij}:=\{-\varepsilon_{i_1}:\,i_1\in \mathcal{I}_{n-i}\}$ for the odd weight case or for the odd-dimensional pair 
$((C_{ij},y_{ij}),(B_{n-i,n-j},x_{ji}))$.

\end{itemize}
Implicitly, the above definition identifies the roots and root vectors which describe each of the blocks $A_{ij}$, $B_{ij}$, $C_{ij}$, $y_{ij}$, $x_{ij}$. When we want to make explicit the group in consideration, we use the upper-script $B,C,D$, and talk for example in the case of a symbol denoting a subset of roots, by $\theta_{ij}^{(B,C,D)}$, $\theta_{ij}^{(B,C,D),c}$, $\theta_{ii}^{C,c}$, $\theta_{ii}^{C,b}$, and $\theta_{ii}^{B,d}$.
From the considerations of section \ref{liestructure} the following proposition follows.
\begin{thm}\label{hodgetriple}
A standard Hodge triple $H_{ij}$ or $H_{ij}^c$ respectively, is equivalent with prescribing a Lie algebra homomorphism of $\mathfrak{sl}(h^{n-i,i}+h^{n-j,j}, \C)$ or $\mathfrak{sl}(h^{i,n-i}+h^{n-j,j})$ respectively to $\g$ such that its Cartan subalgebra $\h$ is mapped to $\h_{ij}$ or $\h_{n-i,j}$ respectively, its positive unipotent radical $\mathfrak{u}^+$ to $A_{ij}$ or equivalently to $B_{ji}$ and its negative unipotent radical $\mathfrak{u}^-$ to $A_{ji}$ or equivalently to $C_{ij}$ and its semisimple part $\mathfrak{s}$ to $D_{ij}$ or equivalently to $D_{n-i,j}$. A nonstandard Hodge triple $H_{ij}^c$, with $n-i=j$, is equivalent with prescribing the natural Lie algebra embedding of $\mathfrak{so}(2h^{i,n-i},\C)$ or $\mathfrak{so}(2h^{i,n-i}+1,\C)$, respectively, for $n$ even and $\mathfrak{sp}(2h^{i,n-i},\C)$ for $n$ odd respectively, to $\g$ respecting the block structure. 
\end{thm}
\begin{proof}
By comparing the root structure for each of the classical Lie algebras as described in \textit{ section } \ref{liestructure} with the root structure describing each of the blocks, it is immediate to see that the Hodge triple $H_{ij}$ describes a diagonal embedding of $sl(h^{n-i,i}+h^{n-j,j}, \C)$ inside $\g$ and the Hodge triple $H_{ij}^c$ an antidiagonal embedding of $sl(h^{i,n-i}+h^{n-j,j},\C)$. 

\end{proof}
\noindent
What we have seen until now is that one can think of the direct sum $\g^{-p,p}\oplus \g^{0,0}\oplus\g^{p,-p}$, for each $p>0$, by replacing if necessary $\g^{0,0}$ by a subspace, as a direct sum of Hodge triples. Thus the Hodge triples are the building blocks of the lie algebra $\g$ when viewed as a Hodge theoretic object, and each Hodge triple is equivalent to a Lie algebra homomorphism of a classical Lie algebra into $\g$. In most of the cases, and in particular for any $\g^{-p,p}$ with $p$ odd, the classical Lie algebra is of type $A$ and thus the name (standard) \textit{Hodge triple}.
\newline
\noindent
\newline
\textbf{Example}. Let $l=6$, the group $G=SO(12,\C)$, the weight $n=4$ and thus $m=2$. It thus follows that the Hodge decomposition contains $5$ members, namely $$V_\C=V^{4,0}\oplus V^{3,1}\oplus V^{2,2}\oplus V^{1,3}\oplus V^{0,4},$$
and the signature sequence is given by $(i^{4-0},i^{3-1},i^{2-2},i^{1-3},i^{0-4})$, which after computing the powers is $+-+-+$.
\noindent
Consider the following set of Hodge numbers: $h^{4,0}=2$, $h^{3,1}=2$ and $h^{2,2}=4$. It thus follows that $G_0=SO(8,4)$ and the period domain is the $G_0$-open orbit parametrized by the signature sequence $+-+-+$ in the flag manifold $G/P$ parametrized by the dimension sequence $(2,2,4,2,2)$. The base point in $D$ is the Hodge filtration associated to the following Hodge decomposition: $V^{4,0}=<e_1,e_2> $, $V^{3,1}=<e_5,e_6>$, $V^{2,2}=<e_3,e_4,e_9,e_{10},>$, $V^{1,3}=<e_{11},e_{12}>$, $V^{0,4}=<e_7,e_8>.$
Since $m=2$, it follows that $\mathcal{J}_1=\{0,2,1\}$ and $\mathcal{J}_2=\{4,2,3\}$ and the blocks $A$, $B$, $C$ and $(-A^t)$ of any element $g\in \mathfrak{so}(12,\C)$ are further divided into $3\times 3$ block matrices.   

$$g=\begin{pmatrix}
\color{blue}{A_{00}} & A_{02} & \color{red}{A_{01}} & B_{04} & B_{02} & \color{green}{B_{03}} \\
A_{20} & A_{22} & A_{21} & B_{24} & B_{22} & B_{23} \\
\color{red}{A_{10}} & A_{12} &\color{blue}{ A_{11}} & \color{green}{B_{14}} & B_{12} & B_{13} \\
C_{40} & C_{42} & \color{green}{C_{41}} & \color{blue}{-A_{44}^t} & -A^t_{42} & -A^t_{43} \\
C_{20} & C_{22} & C_{21} & -A^t_{24} & -A^t_{22} & -A^t_{23} \\
\color{green}{C_{30}} & C_{32} & C_{31} & -A^t_{34} & -A^t_{32} & \color{blue}{-A^t_{11}} 
\end{pmatrix}.$$
In the matrix picture two examples of Hodge triples are depicted, namely the Hodge triple $H_{01}=[A_{01}, (A_{00},A_{11}), A_{10}]$ and $H_{41}^c=[C_{41}, (A_{00},A_{11}), B_{14}]$, with the semisimple part depicted in blue and the unipotent radicals in red and green respectively. The $\g^{-p,p}$ pieces are given by the following blocks written so as to explicitly point out the negative and positive unipotent components: 
\begin{enumerate}
\item{
\begin{itemize}
\item{$\g^{0,0}=A_{00}\oplus A_{22}\oplus A_{11}\oplus C_{22}\oplus B_{22}$,}
\end{itemize}
}
\item{\begin{itemize}
\item{$\g^{-1,1}=A_{10}\oplus A_{21}\oplus C_{21}\oplus C_{32},$}
\item{$\g^{1,-1}=A_{01}\oplus A_{12}\oplus B_{12} \oplus B_{23},$}
\end{itemize}
}
\item{\begin{itemize}
\item{$\g^{-2,2}=A_{20}\oplus C_{42}\oplus C_{20}$,}
\item{$\g^{2,-2}=A_{02}\oplus B_{24}\oplus B_{02}$,}
\end{itemize}
}

\item{\begin{itemize}
\item{$\g^{-3,3}= C_{30}\oplus C_{41}$,}
\item{$\g^{3,-3}= B_{03}\oplus B_{14}$,}
\end{itemize}
}

\item{\begin{itemize}
\item{$\g^{-4,4}= C_{40}$,}
\item{$\g^{4,-4}= B_{04}$.}
\end{itemize}
}
\end{enumerate}
\noindent
\begin{defi}
Let $\g$ be a complex semisimple Lie algebra, $\h$ a Cartan subalgebra and $\theta$ a system of roots of $\g$ with respect to $\h$. A subset of roots $\psi$ of $\theta$ is called commutative if no two roots $\alpha$ and $\beta$ in $\psi$ add up to a root, i.e. $\alpha+\beta\ne$ root, for all $\alpha,\beta\in \psi$.
\end{defi}
\noindent
Note that if $\psi$ is a commutative set of roots and $\mathfrak{a}_\psi$ is the subspace of $\g$ which contains the root vectors of each of the roots in $\psi$, then $\mathfrak{a}_{\psi}$ is abelian, i.e. $[\mathfrak{a}_{\psi},\mathfrak{a}_\psi]=0$.
\begin{lemma}\label{primalema}
The set of roots $\theta_{ij}$, $\theta_{ji}$, $\theta_{ij}^c$ and $\theta_{ji}^c$ describing the unipotent radicals of a standard Hodge triple $H_{ij}$ or $H_{ij}^c$ respectively, are commutative, i.e. $A_{ij}$, $A_{ji}$, $C_{ij}$, $B_{ji}$ are abelian.
\end{lemma}
\begin{proof}
For $l_1,l_2\in \mathcal{I}_{i}$ and $l_3,l_4\in \mathcal{I}_{j}$ we have:
$$(-\varepsilon_{l_1}+\varepsilon_{l_3}) + (-\varepsilon_{l_2}+\varepsilon_{l_4})=
\begin{cases}
 -\varepsilon_{l_1}+\varepsilon_{l_4}, \text{ for } l_2=l_3, l_1\ne l_4\\
 -\varepsilon_{l_3}+\varepsilon_{l_2}, \text{ for } l_1=l_4, l_2\ne l_3\\ 
\text{ not a root otherwise.}
\end{cases}$$
However, because $i\ne j$, it follows that $\mathcal{I}_i\cap\mathcal{I}_j=\emptyset$ and thus only the third case occurs. The statement is clear for $C_{ij}$ and $B_{ji}$ respectively, since $\theta_{ij}^c$ and $\theta_{ji}^c$ respectively contain roots of the form $-\varepsilon_{i_1}-\varepsilon_{i_2}$ and $\varepsilon_{i_1}+\varepsilon_{i_2}$ respectively and for example $(-\varepsilon_{i_1}-\varepsilon_{i_2}) +(-\varepsilon_{i_3}-\varepsilon_{i_4})$ is never a root.
\end{proof}
\noindent
Let $\mathcal{I}_{i,j}:=\mathcal{I}_i\times \mathcal{I}_j$.
\begin{lemma}\label{comutativitate}
Let $\mathcal{I}_{i_1,j_1}$ and $\mathcal{I}_{i_2,j_2}$ be two sets of indices parametrizing roots corresponding to unipotent radicals of standard Hodge triples. Then the set of roots indexed by $\mathcal{I}_{i_1,j_1}$ and $\mathcal{I}_{i_2,j_2}$ is commutative if and only if $i_1\ne j_1\ne  i_2 \ne j_2\ne i_1$.
\end{lemma}
\begin{proof}
In the case of standard Hodge triples involving only $A$-blocks, assume without loss of generality that $i_1<j_1\le m$ and $i_2<j_2\le m$. For $l_1,l_2\in \mathcal{I}_{i_1,j_1}$ and $l_3,l_4\in \mathcal{I}_{i_2,j_2}$ we have on $(A_{j_1,i_1},A_{j_2,i_2})$:
$$(-\varepsilon_{l_1}+\varepsilon_{l_2}) + (-\varepsilon_{l_3}+\varepsilon_{l_4})=
\begin{cases}
 -\varepsilon_{l_1}+\varepsilon_{l_4}, \text{ for } l_2=l_3, l_1\ne l_4\\
 -\varepsilon_{l_3}+\varepsilon_{l_2}, \text{ for } l_1=l_4, l_2\ne l_3\\ 
\text{ not a root otherwise.}
\end{cases}$$
Since all the sets $\mathcal{I}_i$ and $\mathcal{I}_j$ are disjoint for $i\ne j$, it follows that $l_1=l_4\in \mathcal{I}_{i_1}\cap \mathcal{I}_{j_2}$ if and only if $i_1=j_2$ and $l_2=l_3\in \mathcal{I}_{i_2}\cap \mathcal{I}_{j_1}$ if and only if $i_2=j_1$. By swapping $i_1$ and $j_1$ or $i_2$ and $j_2$ the result follows in general.
\newline
\noindent
\newline
In the case of standard Hodge triples involving only $C$- and $B$-blocks, let $i_1\ge m, j_1\le m$, $i_2\ge m, j_2\le m$.
For $l_1,l_2\in \mathcal{I}_{n-i_1,j_1}$ and $l_3,l_4\in \mathcal{I}_{n-i_2,j_2}$ we have on $(C_{i_1,j_1},B_{n-i_1,n-j_1})$ and $(C_{i_2,j_2},B_{n-i_2,n-j_2})$:
$$(-\varepsilon_{l_1}-\varepsilon_{l_2}) + (\varepsilon_{l_3}+\varepsilon_{l_4})=
\begin{cases}
 -\varepsilon_{l_1}+\varepsilon_{l_3}, \text{ for } l_2=l_4, l_1\ne l_3\\
 -\varepsilon_{l_1}+\varepsilon_{l_4}, \text{ for } l_2=l_3, l_1\ne l_4\\
 -\varepsilon_{l_2}+\varepsilon_{l_3}, \text{ for } l_1=l_4, l_2\ne l_3\\
 -\varepsilon_{l_2}+\varepsilon_{l_4}, \text{ for } l_1=l_3, l_2\ne l_4\\ 
\text{ not a root otherwise.}
\end{cases}$$
Since all the sets $\mathcal{I}_i$ and $\mathcal{I}_j$ are disjoint for $i\ne j$, it follows that $l_2=l_4\in \mathcal{I}_{j_1}\cap \mathcal{I}_{j_2}$ if and only if $j_1=j_2$, $l_2=l_3\in \mathcal{I}_{j_1}\cap \mathcal{I}_{n-i_2}$ if and only if $n-i_2=j_1$,$l_1=l_4\in \mathcal{I}_{n-i_1}\cap \mathcal{I}_{j_2}$ if and only if $n-i_1=j_2$, $l_1=l_3\in \mathcal{I}_{n-i_1}\cap \mathcal{I}_{n-i_2}$ if and only if $n-i_1=n-i_2$.
\newline
\noindent
\newline
Finally, consider the case of one standard Hodge triple involving $A$-blocks and one standard Hodge triple involving $C$-and $B$-blocks and assume without loss of generality that $i_1<j_1\le m$, $i_2\ge m$ and $j_2\le m$.  For $l_1,l_2\in\mathcal{I}_{i_1,j_1}$ and $l_3,l_4\in \mathcal{I}_{n-i_2,j_2}$ we have on $(A_{i_1,j_1},C_{i_2,j_2})$: 
$$(-\varepsilon_{l_1}+\varepsilon_{l_2}) + (-\varepsilon_{l_3}-\varepsilon_{l_4})=
\begin{cases}
 -\varepsilon_{l_1}-\varepsilon_{l_4}, \text{ for } l_2=l_3, l_1\ne l_4\\
 -\varepsilon_{l_1}-\varepsilon_{l_3}, \text{ for } l_2=l_4, l_1\ne l_3\\ 
\text{ not a root otherwise.}
\end{cases}$$
It follows that $l_2=l_3\in \mathcal{I}_{j_1}\cap \mathcal{I}_{n-i_2}$ if and only if $j_1=n-i_2$ and $l_2=l_4\in \mathcal{I}_{j_1}\cap \mathcal{I}_{j_2}$ if and only if $j_1=j_2$. On $(A_{i_1,j_1},B_{n-i_2,n-j_2})$ we have: 
$$(-\varepsilon_{l_1}+\varepsilon_{l_2}) + (\varepsilon_{l_3}\varepsilon_{l_4})=
\begin{cases}
 \varepsilon_{l_2}+\varepsilon_{l_4}, \text{ for } l_1=l_3, l_2\ne l_4\\
 \varepsilon_{l_2}+\varepsilon_{l_3}, \text{ for } l_1=l_4, l_2\ne l_3\\ 
\text{ not a root otherwise.}
\end{cases}$$
It follows that $l_1=l_3\in \mathcal{I}_{i_1}\cap \mathcal{I}_{n-i_2}$ if and only if $i_1=n-i_2$ and $l_1=l_4\in \mathcal{I}_{i_1}\cap \mathcal{I}_{j_2}$ if and only if $i_1=j_2$.

\end{proof}

\noindent
\textbf{Hodge brackets of standard Hodge triples}
\newline
\noindent
\newline
Assume that $i_1<j_1$, $i_2<j_2$ and $H_{i_1,j_1}=[A_{i_1,j_1},(A_{i_1,i_1},A_{j_1,j_1}), A_{j_1,i_1}]$, $H_{i_2,j_2}=[A_{i_2,j_2},(A_{i_2,i_2},A_{j_2,j_2}), A_{j_2,i_2}]$ are two standard Hodge triples. As a corollary of \textit{Lemma} \ref{comutativitate}, we have the following bracketing relations:
$$[A_{i_1,j_1},A_{i_2,j_2}]=
\begin{cases}
 A_{i_1,j_2}, \text{ for } j_1=i_2, i_2< j_2\\
 A_{i_2,j_1}, \text{ for } i_1=j_2, i_2< j_1\\ 
\text{ $0$, otherwise.}
\end{cases}$$
$$[A_{j_1,i_1},A_{j_2,i_2}]=
\begin{cases}
 A_{j_2,i_1}, \text{ for } j_1=i_2, i_1< j_2\\
 A_{i_1,i_2}, \text{ for } i_1=j_2, i_2< j_1\\ 
\text{ $0$, otherwise.}
\end{cases}$$
$$[A_{i_1,j_1},A_{j_2,i_2}]=
\begin{cases}
 A_{i_1,i_2}, \text{ for } j_1=j_2, i_2< i_1\\
 A_{j_2,j_1}, \text{ for } i_1=i_2, j_1< j_2\\ 
\text{ $0$, otherwise.}
\end{cases}$$
$$[A_{j_1,i_1},A_{i_2,j_2}]=
\begin{cases}
 A_{i_2,j_1}, \text{ for } j_1=j_2, i_2< i_1\\
 A_{j_1,j_2}, \text{ for } i_1=i_2, j_1< j_2\\ 
\text{ $0$, otherwise.}
\end{cases}$$
Thus we can formally define a \textbf{Hodge bracket}, respecting the usual Lie bracket, of two standard Hodge triples involving $A$-blocks as 

$$[H_{i_1,j_1},H_{i_2,j_2}]:=
\begin{cases}
 H_{i_1,j_2}, \text{ for } j_1=i_2, i_1< j_2\\
 H_{i_2,j_1}, \text{ for } i_1=j_2, i_2< j_1\\
 H_{i_1,i_2}, \text{ for } j_1=j_2, i_1< i_2\\
 H_{j_1,j_2}, \text{ for } i_1=i_2, j_1< j_2\\
\text{ $0$, otherwise.}
\end{cases}$$
and these cases are mutually disjoint so the formal bracket is well-defined. Observe that if $H_{i_1,j_1}$ and $H_{i_2,j_2}$ are standard Hodge triples in $\mathfrak{m}\oplus \g^{0,0}$ and their Hodge bracket is non-trivial, then the Hodge bracket is a Hodge triple in $\k$. This follows from the fact that $(i_1,j_1)$ and $(i_2,j_2)$ have different parities. Thus for example, if $j_1=i_2$ are both odd, then the Hodge bracket is the Hodge triple $H_{i_1,j_2}$ with $i_1$ and $j_2$ even. 
\newline
\noindent
\newline
Assume that $i_1<j_1$, $n-i_2<j_2$ and $H_{i_1,j_1}=[A_{i_1,j_1},(A_{i_1,i_1},A_{j_1,j_1}), A_{j_1,i_1}]$, $H_{i_2,j_2}^c=[A_{i_2,j_2},(A_{i_2,i_2},A_{j_2,j_2}), A_{j_2,i_2}]$ are two standard Hodge triples. As a corollary of \textit{Lemma} \ref{comutativitate}, we have the following bracketing relations:
$$[A_{i_1,j_1},C_{i_2,j_2}]=
\begin{cases}
 C_{n-j_1,j_2}, \text{ for } i_1=n-i_2,\\
 C_{i_2,j_1}, \text{ for } i_1=j_2,\\ 
\text{ $0$, otherwise.}
\end{cases}$$
$$[A_{j_1,i_1},B_{n-i_2,n-j_2}]=
\begin{cases}
 B_{j_1,n-j_2}, \text{ for } i_1=n-i_2, \\
 B_{n-i_2,n-j_1}, \text{ for } i_1=j_2, \\ 
\text{ $0$, otherwise.}
\end{cases}$$
$$[A_{i_1,j_1},B_{n-i_2,n-j_2}]=
\begin{cases}
 B_{i_1,n-j_2}, \text{ for } j_1=n-i_2, \\
 B_{n-i_2,n-i_1}, \text{ for } j_1=j_2,\\ 
\text{ $0$, otherwise.}
\end{cases}$$
$$[A_{j_1,i_1},C_{i_2,j_2}]=
\begin{cases}
 C_{n-i_1,j_2}, \text{ for } j_1=n-i_2, \\
 C_{i_2,i_1}, \text{ for } j_1=j_2, \\ 
\text{ $0$, otherwise.}
\end{cases}$$
Thus we can formally define a \textbf{Hodge bracket}, respecting the usual Lie bracket, of one standard Hodge triple involving $A$-blocks and one standard Hodge triple involving $B$- and $C$-blocks as follows: 
$$[H_{i_1,j_1},H_{i_2,j_2}^c]:=
\begin{cases}
 H_{n-j_1,j_2}^c, \text{ for } i_1=n-i_2, \\
 H_{i_2,j_1}^c, \text{ for } i_1=j_2, \\
 H_{i_2,i_1}^c, \text{ for } j_1=j_2, \\
 H_{n-i_1,j_2}^c, \text{ for } j_1=n-i_2, \\
\text{ $0$, otherwise.}
\end{cases}$$
and these cases are mutually disjoint so the formal bracket is well-defined. Observe that if $H_{i_1,j_1}$ and $H_{i_2,j_2}^c$ are standard Hodge triples in $\mathfrak{m}\oplus \g^{0,0}$ and their Hodge bracket is non-trivial, then the Hodge bracket is a Hodge triple in $\k$. This follows from the fact that $(i_1,j_1)$ and $(i_2,j_2)$ have different parities. Thus for example, if $n=2m+1$, $i_1=n-i_2$ with $i_2$ odd, then $i_1$ is even and thus $j_1$ is odd. Therefore, $n-j_1$ and $j_2$ are both even and the Hodge bracket is the Hodge triple $H_{n-j_1,j_2}^c\in \k$ . 
\newline
\noindent
\newline
Assume that $n-i_1<j_1$, $n-i_2<j_2$ and $$H_{i_1,j_1}^c=[C_{i_1,j_1},(A_{n-i_1,n-i_1},A_{j_1,j_1}), B_{n-i_1,n-j_1}],$$ $H_{i_2,j_2}^c=[C_{i_2,j_2},(A_{n-i_2,n-i_2},A_{j_2,j_2}), B_{n-i_2,n-j_2}]$ are two standard Hodge triples. As a corollary of \textit{Lemma} \ref{comutativitate}, we have the following bracketing relations:
$$[C_{i_1,j_1},B_{n-i_2,n-j_2}]=
\begin{cases}
 A_{j_2,j_1}, \text{ for } n-i_1=n-i_2, j_1<j_2,\\
 A_{n-i_2,j_1}, \text{ for } n-i_1=j_2, j_1<n-i_2,\\ 
 A_{j_2,n-i_1}, \text{ for } j_1=n-i_2, n-i_1< j_2,\\
 A_{n-i_2,n-i_1}, \text{ for } j_1=j_2,n- i_1< n-i_2,\\ 
\text{ $0$, otherwise.}
\end{cases}$$
$$[C_{i_2,j_2},B_{n-i_1,n-j_1}]=
\begin{cases}
 A_{j_1,j_2}, \text{ for } n-i_1=n-i_2, j_1<j_2,\\
 A_{j_1,n-i_2}, \text{ for } n-i_1=j_2, j_1<n-i_2,\\ 
 A_{n-i_1,j_2}, \text{ for } j_1=n-i_2, n-i_1< j_2,\\
 A_{n-i_1,n-i_2}, \text{ for } j_1=j_2,n- i_1< n-i_2,\\ 
\text{ $0$, otherwise.}
\end{cases}$$
Thus we can formally define a \textbf{Hodge bracket}, respecting the usual Lie bracket, of two standard Hodge triples involving $B$- and $C$-blocks as 

$$[H_{i_1,j_1}^c,H_{i_2,j_2}^c]:=
\begin{cases}
 H_{j_1,j_2}, \text{ for } i_1=i_2, j_1< j_2\\
 H_{j_1,n-i_2}, \text{ for } n-i_1=j_2, j_1< n-i_2\\
 H_{n-i_1,j_2}, \text{ for } j_1=n-i_2, n-i_1< j_2\\
 H_{n-i_1,n-i_2}, \text{ for } j_1=j_2, i_2< i_1\\
\text{ $0$, otherwise.}
\end{cases}$$
and these cases are mutually disjoint so the formal bracket is well-defined. Observe that if $H_{i_1,j_1}^c$ and $H_{i_2,j_2}^c$ are standard Hodge triples in $\mathfrak{m}\oplus \g^{0,0}$ and their Hodge bracket is non-trivial, then the Hodge bracket is a Hodge triple in $\k$. This follows from the fact that $(i_1,j_1)$ and $(i_2,j_2)$ have different parities. Thus for example, if $n=2m+1$, $n-i_1=j_2$ with $i_1$ even, then $j_2$ is odd and thus $i_2$ is even. Therefore, $n-i_2$ and $j_1$ are both odd and the Hodge bracket is the Hodge triple $H_{j_1,n-i_2}\in \k$ . 
\newline
\noindent
\newline
We summarize the above results in one Lemma which is of course also a consequence of the more general fact that $[\mathfrak{m},\mathfrak{m}]\subset \k$.
\begin{lemma}\label{bracketcompact}
If $H_{i_1,j_1}$ (respectively $H_{i_1,j_1}^c$) and $H_{i_2,j_2}$ (respectively $H_{i_2,j_2}^c$) are two standard Hodge triples in $\mathfrak{m}^-\oplus\g^{0,0}\oplus\mathfrak {m}^+$ with non-trivial Hodge bracket, then their Hodge bracket is a standard Hodge triple in $\k$.  
\end{lemma}
\noindent
For a standard Hodge triple $H_{ij}$, let $G_{ij}$ be the Lie subgroup of $G$ with Lie algebra $\mathfrak{sl}(h^{n-i,i+h^{n-j,j}},\C)$. Then $G_{ij}$ naturally acts on flag associated to the decomposition $V^{n-i,i}\oplus V^{n-j,j}\oplus V^{j,n-j}\oplus V^{i,n-i}$ with isotropy subgroup $P_{ij}$, the Lie subgroup with Lie algebra $A_{ij}\oplus A_{ii}\oplus A_{jj}$. Then $P_{ij}$ is a maximal parabolic subgroup of $G_{ij}$ and a subgroup of $P$ and $Z_{ij}:=G_{ij}/P_{ij}$ is an equivariently embedded Hermitian symmetric space of compact type of type $A$ in the flag manifold $Z$, i.e. a Grassmannian of $h^{n-i,i}$-planes in $\C^{h^{n-i,i}+h^{n-j,j}}$. Let $her_{ij}$ be the restriction of the Hermitian form $h$ from $V_{\C}$ to $V_{ij}:=V^{n-i,i}\oplus V^{n-j,j}\oplus V^{j,n-j}\oplus V^{i,n-i}$ and let $G_{ij}^0$ the group of isometries of $G_{ij}$ which preserve this Hermitian form. Then, it is immediate that $G_{ij}^0$ is a closed subgroup of $G_0$ and the open $G_{ij}^0$-orbit $D_{ij}^0$ on $G_{ij}/P_{ij}$ parametrized by the signature of the Hermitian form on $V_{ij}$ is the corresponding Hermitian symmetric space of non-compact type and an equivariantly embedded closed submanifold of $D$, i.e. Grassmannian of positive- (or negative-) definite $h^{n-i,i}$-planes in $\C^{h^{n-i,i}+h^{n-j,j}}$. We summarize this in the following lemma:
\begin{lemma}\label{ungrassmann}
The open $D_{ij}^0$-flag domain associated to the Hodge triple $H_{ij}$ is a closed, equivariantely, embedded submanifold of the period domain $D$. 
\end{lemma}
\noindent
For the above example, let $H_{10}=(A_{01},(A_{00},A_{11}),A_{10})$. Then $G_{1,0}=SL(4,\C)$, $G_{1,0}^0=SU(2,2)$. The open orbit can be identified with the Grassmannian of postive $2$-planes in $\C^4$ embedded in its compact dual, the Grassmannian of $2$-planes in $\C^4$.
\begin{defi}
A \textbf{sub-Hodge triple} of a Hodge triple $H_{ij}$ (respectively $H_{ij}^c$), is a formal structure $\tilde{H}_{ij}:=[\tilde{A}_{ij},(\tilde{A}_{ii},\tilde{A}_{jj}),\tilde{A}_{ji}]$ (respectively $\tilde{H}_{ij}^c$) defined as follows. Consider $\tilde{\mathcal{I}}_i$ and $\tilde{\mathcal{I}}_j$, respectively a subset of $\mathcal{I}_i$ and $\mathcal{I}_j$ respectively. Finally, let $\tilde{A}_{ij}$ (resp. $\tilde{C}_{ij}$),  $\tilde{A}_{ii}$ (resp. $\tilde{A}_{n-i,n-i})$, $\tilde{A}_{jj}$ and $\tilde{A}_{ji}$ (resp. $\tilde{B}_{n-i,n-j}$) respectively the elements of $A_{ij}$ (resp. $C_{ij}$), $A_{ii}$ (resp. $A_{n-i,n-i}$), $A_{jj}$ and $A_{ji}$ (resp. $B_{n-i,n-j}$) respectively indexed by $\tilde{\mathcal{I}}_i \times \tilde{\mathcal{I}}_j$,  $\tilde{\mathcal{I}}_i \times \tilde{\mathcal{I}}_i$, $\tilde{\mathcal{I}}_j \times \tilde{\mathcal{I}}_j$ and $\tilde{\mathcal{I}}_j \times \tilde{\mathcal{I}}_i$ respectively.
\end{defi}
\noindent
It is immediate that standard $\tilde{H}_{ij}$, $\tilde{H}_{ij}^c$ (resp. non-standard type $(C)$, $\tilde{H}_{ij}^c$) itself corresponds to an $\mathfrak{sl}(\C)$ (resp. $\mathfrak{sp}(\C)$) embedding in $\g$ whose unipotent radicals and semisimple part can be viewed as sub-blocks of those of $H_{ij}$, $H_{ij}^c$. Furthermore, $\tilde{Z}_{ij}$ and $\tilde{D}_{ij}^0$ respectively are naturally closed sub-Grassmannians of $Z_{ij}$ and $D_{ij}^0$ respectively.
\newline
\noindent
\newline
Let $H_{i_1,j_1}$, $H_{i_2,j_2}$ and $H_{i_3,j_3}$ be three standard Hodge triples such that $j_1=i_2$ and $i_1=j_3$. Then by \textit{Lemma} \ref{bracketcompact} the formal bracket of these Hodge triples lies in $\k$. Thus in order to force the formal bracket to lie in $\mathfrak{m}$ one needs to partition the set of indices into disjoint subsets and pass to sub-Hodge triples.
\begin{thm}\label{teoremaprincipala}
Let $D'=G'_0/P_0'$ be a flag domain of Hodge type and $D=G_0/P_0$ a period domain which are embedded in their compact duals $Z'=G'/P'$ and $Z=G/P$. Then, there exist $\tilde{D}:=D_1\times \dots \times D_s$ a product of irreducible Hermitian symmetric spaces of non-compact type $(A)$ or type $(C)$, equivariantely embedded as closed submanifolds of $D$, such that $D'$ is a closed submanifold of $\tilde{D}$. In particular, $D'$ is a Hermitian symmetric space.    
\end{thm}
\begin{proof}
\textit{Claim $1$:} If $n=2m$ and $\dim V=2l$ (type $(D)$), then all the Hodge triples characterising $\mathfrak{m}^-\oplus\g^{0,0}\oplus\mathfrak{m}^+$ are standard Hodge triples and thus correspond to closed, equivariantely embedded, irreducible Hermitian symmetric spaces of type $(A)$. 
\newline
\noindent
\newline
To see this observe that for fixed $p>0$, $p$ odd the statement is clear for the blocks characterising $m\cap A$. Assume $H_{ij}^c=(C_{ij}, (A_{n-i,n-i},A_{j,j}), B_{ji}) $ is a Hodge triple describing $(\mathfrak{m}\cap C) \oplus \g^{0,0}$, with $i\ge m$, $j\le m$. Then, we need to show that $n-i\ne j$. But this follows from the fact that, $i-j=p$ and thus $i$ and $j$ have different parities. Together with the fact that $n$ is even, this further implies that $n-i$ and $j$ have different parities and thus they cannot be equal.  
\newline
\noindent
\newline
\textit{Claim $2$:} If $n=2m+1$ and $\dim V=2l$ (type $(C)$), then all the Hodge triples characterising $\mathfrak{m}^-\oplus\g^{0,0}\oplus\mathfrak{m}^+$ are either standard Hodge triples or non-standard Hodge triples of type $(C)$ and thus correspond to closed, equivariantely embedded, irreducible Hermitian symmetric spaces of type $(A)$ and of type $(C)$.
\newline
\noindent
\newline
To see this first observe that for fixed $p>0$, $p$ odd it is clear that the Hodge triples characterising $(m \cap A) \oplus \g^{0,0}$ are standard. Assume $$H_{ij}^c=(C_{ij}, (A_{n-i,n-i},A_{j,j}), B_{ji}) $$ is a Hodge triple describing $(\mathfrak{m}\cap C) \oplus \g^{0,0}$ with $i\ge m$, $j\le m$. Now as before $i-j=p$ and thus $i$ and $j$ have different parities. However, since $n$ is odd, there exist an $i$ and a $j$ such that $n-i=j$ and $i-j=p$, namely $j=(n-p)/2$. Thus for any $p$ odd one obtains a non-standard Hodge triple of type $(C)$, with $C_{ij}$ one of the diagonal blocks in $C$.
\newline
\noindent
\newline
\textit{Claim $3$:} If $n=2m$ and $\dim V=2l+1$ (type $(B)$), then all the Hodge triples characterising $(\mathfrak{m}-({y\oplus x}))\oplus {\g^{0,0}}$ are standard Hodge triples and thus correspond to closed, equivariantely embedded, irreducible Hermitian symmetric spaces of type $(A)$. Furthermore, any non standard Hodge triple of type $B$ involving $\mathfrak{m} \cap y$ has non-empty intersection with $\k$, unless it corresponds to an equivariently embedded $\mathfrak{so}(3,\C)\cong \mathfrak{su}(2)$.
\newline
\noindent
\newline
To see this observe that as for type $(D)$ for fixed $p>0$, $p$ odd the statement is clear for the blocks characterising $m\cap A$. Assume $H_{ij}^c=(C_{ij}, (A_{n-i,n-i},A_{j,j}), B_{ji}) $ is a Hodge triple describing $(\mathfrak{m}\cap C) \oplus \g^{0,0}$, with $i\ge m$, $j\le m$. Then, we need to show that $n-i\ne j$. But this follows, as for type $(D)$, from the fact that, $i-j=p$ and thus $i$ and $j$ have different parities. Together with the fact that $n$ is even, this further implies that $n-i$ and $j$ have different parities and thus they cannot be equal.
\newline
\noindent
\newline
For fixed $p$, the block $\mathfrak{m}\cap y=y_{m, m-p}$ is characterized by the set of roots $\theta_{m,m-p}^d=\{-\varepsilon_{i_1}: i_1\in \mathcal{I}_{m-p}\}$. Thus for any $i_1\ne i_2$ in $\mathcal{I}_{m-p}$, $(-\varepsilon_{i_1})+(-\varepsilon_{i_2})=-\varepsilon_{i_1}-\varepsilon_{i_2}$ is a root in $\theta^c_{m+p,m-p}$. Thus any two distinct roots in $\theta^d_{m,m-p}$ add up to a root describing a root vector in $\mathfrak{k}$ and more precisely in $\g^{-2p,2p}$. Thus the only possible Hodge triple involving ${y}_{m,m-p}$ in $\mathfrak{m}$ is $[{y}_{m,i_1},a_{i_1,i_1},x_{i_1,m}]$ for some fixed $i_1\in \mathcal{I}_{m-p}$ and thus an equivariently embedded $\mathfrak{so}(3,\C)\cong \mathfrak{su}(2)$.  
\newline
\noindent
\newline
It thus follows that in each case, i.e. type $(B)$, $(C)$, or $(D)$, by passing to sub-Hodge triples whenever necessary, one obtains closed, equivariantely embedded Hermitian symmetric spaces which are products of irreducible Hermitian symmetric spaces of non-compact type $(A)$ or, for the odd weight case, of non-compact type $(C)$. Any equivariantely embedded Mumford-Tate domain must be a submanifold of such a Hermitian symmetric space $\tilde{D}$ and thus itself a Hermitian symmetric space. Furthermore, this means that $G'$ is a closed subgroup of some $\tilde{G}$ and thus $G_{0}'$ is a closed subgroup of $\tilde{G}_0$. Thus since $\tilde{G}_0$ acts properly on $\tilde{D}$ the action of $G_0'$ on $D'$ is also proper and thus if $z_0$ is the base point of the embedding of $D'$ in $\tilde{D}$, it follows that $D'=G_0'.z_0$ is closed in $\tilde{D}$.
\end{proof}
\noindent
\textbf{Remark1.} Observe that the assumption that the flag domains we consider are orbits of groups of Hodge type is essential. For example in the case of a Grassmannian of type $A$ associated to a standard Hodge triple, this assumption insures that the equivariantly embedded flag domain $\tilde{D}$ which contains the base point $z_0$ is an orbit of $SU(p,q)$ for some $p$ and $q$. It would be interesting to see what happens if one allows $\tilde{D}$ to be a flag domain of $SL(n,\R)$ or $SL(n,\mathbb{H})$, the other two real forms (up to conjugation) of $SL(n,\C)$.
\newline
\noindent
\newline
\textbf{Remark2.} As a future project it will be interesting to identify transversal complex manifolds to the base cycle which can  arize as closed complex submanifolds of Mumford-Tate domains, without requiring the more strong condition $2.$ considered in this paper as part of the definition of a flag domain of Hodge type. 
\begin{cor}\label{clasicalgroups}
The flag domains of Hodge type are Hermitian symmetric spaces of non-compact type whose irreducible factors are of type (A), (B), (C), and (D), i.e. no exceptional cases appear. Furthermore, the embeddings describing each of the irreducible factors are precisely the ones listed by Satake in \cite{Satake}.
\end{cor}
\begin{proof}
From \textit{Theorem} \ref{teoremaprincipala} it follows that the classification problem (*) of classifying embeddings $f:D'\rightarrow D$ from an arbitrary flag domain of Hodge type to an arbitrary period domain $D$ is reduced to the classification problem (*) of classifying embeddings $f:D'\rightarrow D''$ from a Hermitian symmetric space of Hodge type to $D''$ a Hermitian symmetric space whose irreducible factors are of type $(A)$ or type $(C)$. But these are precisely the conditions required in the classification problem considered by Satake in \cite{Satake}.
\end{proof}
\noindent
The next proposition restates the general results  concerning the $\g^{-p,p}$-block characterisation, for the subspace $\g^{-1,1}$ to explicitly point out the identification of each of its blocks with a tangent space to a certain Grassmannian, i.e. with a unipotent radical, and recall the explicit description of the Harisch-Chandra coordinates characterising each block.
\begin{prop}\label{mainproposition}
The subspace $\g^{-1,1}$ can be decomposed as the following direct sum of vector subspace: $$\g^{-1,1}_B=\mathfrak{u}_{1,0}^B+\mathfrak{u}^B_{2,1}+\dots +\mathfrak{u}_{m,m-1}^B+\mathfrak{u}_{m,m-1}^{B,c}+\mathfrak{y}_{m,m-1},$$
$$\g^{-1,1}_C=\mathfrak{u}_{1,0}^C+\mathfrak{u}^C_{2,1}+\dots +\mathfrak{u}_{m,m-1}^C+\mathfrak{u}_{m+1,m}^{C,c},$$
$$\g^{-1,1}_D=\mathfrak{u}_{1,0}^D+\mathfrak{u}^D_{2,1}+\dots +\mathfrak{u}_{m,m-1}^D+\mathfrak{u}_{m,m-1}^{D,c}.$$
Each member $\mathfrak{u}_{k+1,k}^{(B,C,D)}$, for all $0\le k \le m-1$, $\mathfrak{u}_{m-1,m}^{(B,D),c}$ and $\mathfrak{u}_{m+1,m}^{(C,c)}$ respectively, are expressed in terms of Harisch-Chandra coordinates by the set of root vectors $\theta_{k+1,k}^{(B,C,D)}$, for all $0\le k \le m-1$, $\theta_{m,m-1}^{(B,D),c}$ and $\theta_{m-1,m}^{C,c}$ respectively and can be naturally identified with the tangent space at a base point of a Grassmannian of type A, for the first two type of subspaces, and a Grassmannian of type C for the last one, respectively. The subspace $\mathfrak{y}_{m,m-1}$ corresponds to the set of roots $\theta_{m,m-1}^{B,d}$. 

\end{prop}
\begin{proof}
To better suggest that the blocks $A_{j+1,j}$ can be identified with the unipotent radical (or opposite unipotent radical) of the maximal parabolic subgroup of $sl(h^{n-j,j}+h^{n-(j+1),j+1}, \C)$ as defined in section \ref{liestructure}, associated to the diagonal  embedding given by the Hodge triple $H_{j,j+1}$ as stated in proposition \ref{hodgetriple}, we change the notation to $\mathfrak{u}_{j+1,j}$ and use the upper scripts $B$, $C$, $D$ to indicate the Lie group type. 
\newline
\noindent
\newline
The notation $\mathfrak{u}_{m,m-1}^{(B,D),c}$ and $\mathfrak{u}_{m+1,m}^{C,c}$ respectively, replaces the notation $C_{m,m-1}$,  and $C_{m+1,m}$ respectively. This corresponds to the antidiagonal embedding of $$sl(h^{m+1,m-1}+h^{m,m},\C)$$ in $\g$ given by the Hodge triple $H_{m-1,m}$ and to the natural embedding of $\mathfrak{sp}(2h^{m,m},\C)$ in $\g$ given by the Hodge triple $H_{m+1,m}^c$ respectively. 
\newline
\noindent
\newline
Any embedding of $\mathfrak{so}(h,\C)$ inside $\g$ involving the subset $\mathfrak{y}_{m,m-1}$ would have non-empty intersection with $\g$ outside the subspace $\g^{-p,p}$, p odd, i.e. it will contain compact factors, unless $h=3$. 
\end{proof}

\begin{cor}\label{dimensionformula}
The dimension of $\g^{-1,1}$ is given in each case by: $$\dim\g^{-1,1}_B=f_0f_1+f_1f_2+\dots + f_{m-1}(f_m-1)/2+f_{m-1}(f_m-1)/2+f_{m-1}, $$ $$\dim\g^{-1,1}_C=f_0f_1+\dots f_{m-1}f_m+f_m(f_m+1)/2,$$ $$\dim\g^{-1,1}_D=f_0f_1+f_1f_2\dots +f_{m-1}f_m/2+f_{m-1}f_m/2,$$
where $f_kf_{k+1}$ for all $0\le k \le m-2$ and $f_{m-1}(f_m-1)/2$ is the dimension of the naturally embedded Grassmanian of type (A), and $f_{m}(f_{m+1}+1)/2$ is the dimension of the naturally embedded Grassmanian of type (C).  
\end{cor}
\noindent
In general, one obtains as a corollary of Proposition \ref{orthogonaleven} 
 the following dimension formula, for any $\g^{-p,p}$-piece, with $p$ odd, in the even weight case, even dimension (type D) and one can write similar formulas in general.
\begin{cor}\label{generaldimensionformula}
The dimension of $\g^{-p,p}$ for $0<p<m$ is given by: $$\sum_{j=0}^{m-p-1} f_{j+p},f_j+(1/2)f_mf_{m-p}+(1/2)f_mf_{m-p}+(1/2)\sum_{j=m-p+1}^{m-1} f_{n-j-p}f_j, $$
and for $p>m$ it is given by $$\dim \g^{-p,p}=(1/2)\sum_{j=0}^{n-p}f_{n-p-j}f_j.$$
\end{cor}
\noindent
\textbf{Abelian subspaces of $\g^{-1,1}$.}
\newline
\noindent
\newline
Next we give a procedure of constructing abelian subspaces of $\g^{-1,1}$ using the described properties of Hodge triples and Hodge brackets and show that each abelian subspace of $\g^{-1,1}$ is contained in one of the abelian subspaces prescribed by this algorithm, i.e. a procedure for describing sub-Hodge triples.
\begin{defi}
For $n=2m+1$ and for each $1\le k\le m$, call any partition of $\mathcal{I}_k=\mathcal{I}_{k}^1\cup \mathcal{I}_k^2$, a \textbf{Hodge path} and the sequence $(j_0,j_1,\dots, j_m)$, where $j_0:=0$, $\mathcal{I}_0^1=\mathcal{I}_0^2=\mathcal{I}_0$, $0\le j_k:=|\mathcal{I}_k^1|\le f_k$ and $j_{k+1}=0$ whenever $h^{n-k,k}-j_k=0$, a \textbf{Hodge sequence}. If $n=2m$ consider for type $(D)$ the same together with the extra requirement that $\mathcal{I}_{m-1}=\mathcal{I}_{m-1}^1\cup \mathcal{I}_{m-1}^2\cup \mathcal{I}_{m-1}^3$ and $\mathcal{I}_{m}=\mathcal{I}_{m}^1\cup \mathcal{I}_{m}^2$ and for type $B$ for any choice of element in $y_{m,m-1}$ one excludes the corresponding index in $\mathcal{I}_{m-1}$.
\end{defi}
\noindent
 $\bullet$ \textbf{Description of algorithm:} Following a Hodge path construct a subspace $\g^{-1,1}_{\mathfrak{a}}$ of $\g^{-1,1}$, by letting first $\mathfrak{u}_{j_1,j_0}$ be the subspace of $\mathfrak{u}_{1,0}$ composed of rows corresponding to indices in $\mathcal{I}_1^1$. Then construct inductively $\mathfrak{u}_{j_{k+1},j_k}$ as the subspace of $\mathfrak{u}_{k+1,k}$ of columns corresponding to indices in $\mathcal{I}_k^2$ and rows corresponding to indices in $\mathcal{I}_{k+1}^1$, for all $1\le k \le m-1$. Finally set $\mathfrak{u}^c_{j_{m},j_{m-1}}$ to be the subspace of $\mathfrak{u}_{m,m-1}^c$ of rows corresponding to $i+l$, for all $i \in \mathcal{I}_{m}^2$ and columns corresponding to $\mathcal{I}_{m-1}^3$. For $\dim V=2m+1$, if one chooses a root vector of $\theta_{m,m-1}^{B,d}$, then the root vectors in $\mathfrak{u}_{m-1,m}$, $\mathfrak{u}_{m,m-1}^c$ and $\mathfrak{u}_{m-2,m-1}$ which have one index the same as the root vector of $\theta_{m,m-1}^{B,d}$ are eliminted and then one starts the algorithm. 
\newline
\noindent
\newline
Reasoning as for \textit{Theorem} \ref{hodgetriple}, one sees that each member $\mathfrak{u}_{j_{k+1},j_{k}}$ is expressed in terms of Harish-Chandra coordinates  by root vectors corresponding to $$\theta_{j_{k+1},j_{k}}=\{-\varepsilon_i+\varepsilon_j:\, i \in \mathcal{I}_{k+1}^1, \, j\in \mathcal{I}_{k}^2\},$$
for all $1\le k \le {m-1}$ and can be naturally identified with a Grassmannian of type A, of dimension $(f_k-j_k)f_{k+1}$. Likewise, $u_{j_{m},j_{m-1}}^c$ for $n=2m$ is expressed in terms of Harisch-Chandra coordinates by $$\theta_{j_{m},j_{m-1}}^c=\{-\varepsilon_i-\varepsilon_j:\, i \in \mathcal{I}_{m-1}^3,\, j\in \mathcal{I}_m^2\}$$ and naturally identified with a Grassmannian of type A of dimension $(f_{m-1}-j_{m-1}-j_{m-2})(f_m-j_m)$. For $n=2m+1$,  $u_{j_{m},j_{m-1}}^c$ is expressed in terms of Harisch-Chandra coordinates by $$\theta_{j_{m},j_{m-1}}^c=\{-\varepsilon_i-\varepsilon_j:\, i, j\in \mathcal{I}_m^2\}\cup\{-\varepsilon_i:\,i\in \mathcal{I}_{m}^2\}$$ and naturally identified with a Grassmannian of type C of dimension $(f_m-j_m)(f_m-j_m+1)/2$.  Thus, as a special case of \textit{Theorem} \ref{teoremaprincipala}, we have:
\begin{thm}\label{maintheorem}
Let $\mathfrak{a}$ be an abelian subalgebra of $\g^{-1,1}$. Then there exist a Hodge path and a Hodge sequence depending on $\mathfrak{a}$ such that $\mathfrak{a}$ is a Lie subalgebra of $\g^{-1,1}_{\mathfrak{a}}$. Furthermore, the subspace $\g^{-1,1}_{\mathfrak{a}}$ can be naturally identified with the tangent space at a base point in $D$ to an equivariently-embedded product of Grassmannian submanifolds $G_1\times G_2\times \dots \times G_s$. For even weight each $G_i$ is of type (A), whereas for odd weight each $G_i$ is of type (A) with one possible member of type (C).
\end{thm}
\begin{cor}
Consider the following dimension function $f:[0,d_1]\times[0,d_2]\times\dots\times[0,d_{m-1}]$, defined by $$f(l_1,l_2,\dots,l_{m-1})=d_0l_1+d_1l_2+\dots d_{m-1}d_m-l_1l_2-l_2l_3-\dots -l_{m-1}d_m.$$
The dimension of a maximal abelian subalgebra is the maximum value of the function $f$.
\end{cor}
\begin{proof}
See the below remark and page 13 of \cite{VHS}.
\end{proof}
\noindent
\newline
\noindent
\textbf{Abelian subspaces of $\g^{-p,p}$ and $\mathfrak{m}$}
\newline
\noindent
\newline
Let $\g=\k\oplus\mathfrak{m}$ be the natural Cartan decomposition on $\g$ defined by our choice of base Hodge structure on $D$. In this subsection we give a procedure of constructing abelian subspaces of $\g^{-p,p}$ for $p$ odd, and of $\mathfrak{m}=\oplus_{p \equiv 1 (\text{mod } 2)}\g^{-p,p}$, respectively suited for combinatorial considerations. Assume without loss of generality that $m$ is even, the case $m$ odd being analogous.
\newline
\noindent
\newline
For a matrix $M$, call the $i^{th}$-diagonal, the matrix with only nonzero elements on the spots starting on the $i^{th}$-row, first column or first row, $i^{th}$-column and moving one step down diagonally from left to right. Call the $i^{th}$-antidiagonal the matrix with only nonzero elements on the spots starting on the $i^{th}$-row, $i<m$, first column or last row, $i^{th}$-column and moving one step up diagonally from left to right.  
\newline
\noindent
\newline
Following the notation of \textit{section} \ref{combinatorics}, let $\mathcal{J}_1:=\mathcal{J}_1^1\cup \mathcal{J}_1^2$, $\mathcal{J}_2:=\mathcal{J}_2^1\cup \mathcal{J}_2^2$, be the partition of $\mathcal{J}_1$ and $\mathcal{J}_2$ into even and odd integers, respectively. Then in the even weight case the blocks of $\mathfrak{m}\cap A$, $\mathfrak{m}\cap B$,$\mathfrak{m}\cap C$ and $\mathfrak{m}\cap (-A^t)$ respectively, are indexed by $(\mathcal{J}_1^2\cup \mathcal{J}_1^1)\cup (\mathcal{J}_1^1\cup \mathcal{J}_1^2)$, $(\mathcal{J}_1^1\cup \mathcal{J}_2^2)\cup (\mathcal{J}_1^2\cup \mathcal{J}_2^1)$,  $(\mathcal{J}_2^1\cup \mathcal{J}_1^2)\cup (\mathcal{J}_2^2\cup \mathcal{J}_1^1)$ and $(\mathcal{J}_2^1\cup \mathcal{J}_2^2)\cup (\mathcal{J}_2^2\cup \mathcal{J}_2^1)$, respectively. In the odd weight case the blocks of $\mathfrak{m}\cap A$, $\mathfrak{m}\cap B$,$\mathfrak{m}\cap C$ and $\mathfrak{m}\cap (-A^t)$ respectively, are indexed by $(\mathcal{J}_1^2\cup \mathcal{J}_1^1)\cup (\mathcal{J}_1^1\cup \mathcal{J}_1^2)$, $(\mathcal{J}_1^1\cup \mathcal{J}_2^1)\cup (\mathcal{J}_1^2\cup \mathcal{J}_2^2)$,  $(\mathcal{J}_2^1\cup \mathcal{J}_1^1)\cup (\mathcal{J}_2^2\cup \mathcal{J}_1^2)$ and $(\mathcal{J}_2^1\cup \mathcal{J}_2^2)\cup (\mathcal{J}_2^2\cup \mathcal{J}_2^1)$, respectively.
\newline
\noindent
\newline
If $0<p<m$, it follows from \textit{Section} \ref{combinatorics} that $\g^{-p,p}$ contains both $A$- and $C$-blocks, whereas if $p>m$, then $\g^{-p,p}$ contains only $C$-blocks. Recalling the indexing notation of \textit{Section} \ref{combinatorics}, these blocks are the following diagonals: 
\begin{itemize}
\item{$p^{th}$-diagonal of $m\cap \mathfrak{b}^-$,}
\item{$(p+1)^{th}$-diagonal of $\mathfrak{m}\cap \mathfrak{b}^+$,}
\item{$(p+1)^{th}$-antidiagonal of $\mathfrak{m}\cap C$ for $p>m$ and $(m-p)^{th}$-antidiagonal of $\mathfrak{m}\cap C$, for $p<m$,}
\end{itemize}
where $\b^-$ and $\b^+$ are the Borel subalgebras of lower triangular and upper triangular matrices, respectively, of the block $A$, viewed as the diagonal embedding of $\sl$ in $\g$. 
\newline
\noindent
\newline
Note that the set of indices $\mathcal{I}_{i,j}:=\mathcal{I}_i\times \mathcal{I}_j$ parametrising blocks of $m$ have the following properties: $i>j$, $i-j=p$ for some $p$ odd and $j\le m$.
\newline
\noindent
\newline
\noindent
\textbf{Algorithm for describing abelian subspaces of $\g^{-p,p}$ for fixed $p$ odd.}
\newline
\noindent
\newline
\textbf{If $p>m$}, then the $\g^{-p,p}$-piece is identified with the $(p+1)^{th}$-antidiagonal of $\mathfrak{m}\cap C$, whose blocks are indexed by the pairs $\{(p+1,1), (p+3,3), \dots, (p+m-1,m-1)\}$. From \textit{Lemma} \ref{comutativitate} it immediately follows that this forms a commutative set of roots and no partition is needed. 
\newline
\noindent
\newline
\textbf{If $p<m$}, let $m=k\cdot p+r$, such that $k\ge1$ and $r<p$. The blocks parametrising $\g^{-p,p}\cap A$ are indexed by the pairs $\{(p,0), (p+1,1), \dots, (m,m-p)\}$. If $k=1$, then from \textit{Lemma} \ref{comutativitate} it follows that this set parametrises a commutative set of roots and no partition is required. If $k>1$, then for $j\in \{0,1,\dots, r\}$, we group these pairs into sets of $k$-elements, namely $$\mathcal{A}_j:=\{(p+j,j),(2p+j,p+j),(3p+j,2p+j),\dots, (kp+j,(k-1)p+j)\}.$$ For $r<j\le p-1$, we group the rest of the pairs into sets of $k-1$ elements: $$\mathcal{B}_j:=\{(p+j,j), (2p+j,p+j),\dots, (k-1)p+j,(k-2)p+j\}.$$ 
\noindent
One can think of $k$, $r$ and $p$ as certain quantities which count the frequency with which commuting blocks appear. Observe that for $p=1$ we have $j=0$ and only one such set appears. From \textit{Lemma} \ref{comutativitate}, a root indexed by a pair in $\mathcal{A}_{j}$ and a root indexed by a pair in $\mathcal{B}_{j}$ will not add up to a root. Also for different values $j_1$ and $j_2$ a root indexed by a pair in $\mathcal{A}_{j_1}$ together with a root indexed by a pair in $\mathcal{A}_{j_2}$ will not add up to a root and likewise for pairs in $\mathcal{B}_{j_1}$ and $\mathcal{B}_{j_2}$.
\newline
\noindent
\newline
For each $j$ a partition of $\mathcal{I}_{fp+j}:=\mathcal{I}_{fp+j}^1\cap \mathcal{I}_{fp+j}^2$ for all $1\le f\le k-1$ will produce a disjoint set of indices parametrising a commutative set of roots. At the level of blocks this procedure corresponds first to the choice of subspace of $A_{p+j,j}$ composed of rows corresponding to indices in $\mathcal{I}_{p+j}^1$. Then it constructs inductively the subspace of $A_{fp+j,(f-1)p+j}$ of columns corresponding to indices in $\mathcal{I}_{(f-1)p+j}^2$ and rows corresponding to indices in $\mathcal{I}_{fp+j}^1$. As before the roots indexed by elements of $\mathcal{I}_{(f-1)p+j}^2\times \mathcal{I}_{fp+j}^1$ are the Harish-Chandra coordinates of this unipotent piece. 
\newline
\noindent
\newline
The blocks parametrizing $C\cap \g^{-p,p}$ are indexed by the pairs $\{(m,m-p), (m+2,m-p+2), \dots , (m+p-1,m-1)\}$ and corresponding roots are indexed by $\{(n-m,m-p), (n-m-2,m-p+2), \dots , (n-m-p+1, m-1)\}$. From \textit{Lemma} \ref{comutativitate} it follows that this set parametrizes a commutative set of roots so no partition is needed. However, the indexing set for the roots is the same as that indexing the roots involved in the $(m-p)^{th}$-antidiagonal of $\mathfrak{m}\cap\b^{+}$ so a further partition would be needed for the pairs that parametrize roots which together with roots of $\mathcal{A}_j$ and $\mathcal{B}_j$ will add up to a root.
\newline
\noindent
\newline
\newline
\noindent
\newline
\textbf{Remark1. Schubert variations of Hodge structure.}
In \cite{Kerr1} and \cite{Robles}, Kerr and Robles consider flag manifolds $Z=G/P$ which arise as compact duals of Mumford-Tate domains. One of the motivating question of the authors work is which Schubert varieties (i.e. closure of B-orbits) in Z satisfy the infinitesimal period relation (these are called by the authors Schubert variations of Hodge structure or horizontal Schubert varieties) and further, which of those are smooth, homogeneous submanifolds of $Z$. The main result of the authors in relation with this question is \textbf{Theorem} $1.3.$ in \cite{Kerr1} which states the following:
\textit{Let $X\subset Z=G/P$ be a horizontal Schubert variety. If $X$ is smooth, then $X$ is a product of homogeneously embedded, rational homogeneous subvariaties $X(\mathcal{D}')\subset G/P$ corresponding to subdiagrams $\mathcal{D}'\subset \mathcal{D}$ (of the Dynkin diagram associated to $P$). Moreover, each $X(\mathcal{D'})$ is cominuscule.} We believe that using the main theorem of the current paper \textit{Theorem} \ref{teoremaprincipala} and the techniques used to describe Schubert variaties in \cite{Kerr1} and \cite{Robles} one can generelize their theorem by replacing the requirement that a Schubert variety must satisfy the infinitesimal period relation with the more general condition that the tangent space of the Schubert variety at a base point can be identified with a subspace of $\mathfrak{m}\cap\mathfrak{u}^-$, i.e. it satisfies the \textit{condition} $2.$ in the classification problem (*) and obtain the same result. Thus one can ask which Schubert varieties appear as compact duals of flag domains of Hodge type as defined by the classification problem (*). In any case, if these Schubert varieties or any of the Schubert variations of Hodge structure do appear as compact duals of Mumford-Tate domains, then using the results of the current paper one can identify their corresponding Mumford-Tate domain. 
\newline
\noindent
\newline
\textbf{Remark2.} In \cite{VHS}, Carlson, Toledo, and Kasparian investigate the problem of  finding an upper bound for the rank of holomorphic maps $f:N \rightarrow D$, with $D$ a period domain, satisfying the infinitesimal period relation. Let us briefly outline their results (here in boldface) in the context of the current paper. To do this the authors first reduce the problem to finding an upper bound for the dimension of any abelian subspace $\g^{-1,1}$ which can be written as a direct sum of root vectors corresponding to a set of positive, commutative roots (see subsection \textbf{2.Abelian subspaces}). In order to describe this maximal abelian subspaces of $\g^{-1,1}$, as in the current paper, the authors analyse the action of an element $g\in \g^{-1,1}$, as an endomorphism, on the base Hodge decomposition. In contrast to the current paper, the authors choose a Hodge frame, i.e. a basis of the vector space adapted to the base Hodge decomposition, to give a block decomposition of $g\in \g^{-1,1}$. This has the advantage that, using notation from the current paper, if one indexes the blocks of $g$ from $0$ to $n+1$, in the odd weight case and $n+2$ in the even weight case, then the blocks on the $p^{th}$-lower diagonal describe the $\g^{-p,p}$ piece for $p>0$ (see subsection \textbf{4. Block decomposition}). Moreover, after choosing a system of positive roots, the blocks of $\g^{-1,1}$, which are blocks siting right bellow the main diagonal, will be described by a set of negative roots and \textbf{Theorem 5.4} contains a description of the set of roots corresponding to each $\g^{-p,p}$ piece. From the block description the dimension of $\g^{-1,1}$ becomes immediately clear.  In subsection \textbf{6. Proof of the bound}, the upper bound is computed by computing the maximal value of the function $f$ described in \textbf{Corollary 8}. of the current paper. In subsection \textbf{7. Sharpness} the paper contains examples of variations of Hodge structure that achieve this maximal bound. 
\noindent
\newline
\noindent
The problem of finding the block structure of each $\g^{-p,p}$ in a Hodge frame is also suggested in \cite{Stefan} on page $325$, exercise $12.1.4$. In the current paper we were interested in classifying all embeddings of flag domains in a period domain satisfying infinitesimal period relation and more generally local transversality with respect to a base cycle, thus emphasizing the transversal homogeneous spaces as opposed to just the tangent vector space. Consequently, we did not choose an adapted Hodge frame as a basis but a basis on which the root structure of $\g$ is most easily understood and the block structure of the $\g^{-p,p}$ pieces perfectly aligns with the canonically defined Cartan decomposition . As a consequence we have not only described the blocks as vector subspaces of $\g^{-p,p}$ together with their corresponding root vectors, but have shown how this root vectors are actual Harisch-Chandra coordinates for certain $\mathfrak{sl}$, $\mathfrak{so},$ or $\mathfrak{sp}$ embeddings. As such this leads to a complete classification of embeddings of flag domains in $D$ satisfying IPR or the local transversality with respect to the base cycle condition. 
\newline
\noindent
\newline
\textbf{Contact:}\\
{Ana-Maria Brecan}\\  
{anabrecan@gmail.com}\\
Institut f{\"u}r Mathematik \\
FB 08 Physik, Mathematik und Informatik \\
Johannes Gutenberg-Universit{\"a}t Mainz\\
Staudingerweg 9 \\
55099 Mainz\\
Germany.\\

\newpage
\bibliographystyle{alphanum}
\bibliography{citations} 

\begin{thebibliography}{CMSP}

\bibitem[BA]{Bell}
R.~B. Bell and J.~L. Alperin.
\newblock {\em {Groups and Representations}}, volume 162 of {\em Graduate Texts
  in Mathematics}.
\newblock Springer-Verlag, New York, 1995.

\bibitem[CKT]{VHS}
J.~Carlson, A.~Kasparian, and D.~Toledo.
\newblock {Variation of Hodge structure of maximal dimension}.
\newblock {\em Duke Mathematical Journal}, 58(3):669--694, 1989.

\bibitem[CMSP]{Stefan}
J.~Carlson, S.~M\"{u}ller-Stach, and Chris Peters.
\newblock {\em {Period mappings and Period Domains}}, volume~85 of {\em
  Cambridge studies in advanced mathematics}.
\newblock Cambridge University Press, Boston, 2003.

\bibitem[FHW]{Fels2006}
G.~Fels, A.~Huckleberry, and J.~A. Wolf.
\newblock {\em {Cycle Spaces of Flag Domains}}, volume 245 of {\em Progress in
  Mathematics}.
\newblock Birkh\"{a}user-Verlag, Boston, 2006.

\bibitem[FL]{Laza}
R.~Friedman and R.~Laza.
\newblock {Semi-algebraic horizontal subvarieties of Calabi-Yau type}.
\newblock {\em Duke Mathematical Journal}, 162(12):2077--2148, 2013.

\bibitem[GGK]{MTbook}
M.~Green, P.~Griffiths, and M.~Kerr.
\newblock {\em {Mumford-Tate groups and domains}}.
\newblock Number 183 in Annals of Mathematics Studies. Princeton University
  Press, 2012.

\bibitem[GW]{Wallach}
R.~Goodman and N.~Wallach.
\newblock {\em {Symmetry, Representations and Invariants}}, volume 255 of {\em
  Graduate Texts in Mathematics}.
\newblock Springer Verlag, New York, 2009.

\bibitem[Kna]{Knapp2002}
A.~W. Knapp.
\newblock {\em {Lie Groups Beyond an Introduction}}.
\newblock Progress in Mathematics. Birkh\"{a}user Boston, Boston, 2 edition,
  2002.

\bibitem[KR1]{Kerr1}
M.~Kerr and C~Robles.
\newblock {Classification of smooth horizontal Schubert varieties}.
\newblock 2016.

\bibitem[KR2]{Kerr2}
M.~Kerr and C~Robles.
\newblock {Variations of Hodge structure and orbits in flag variaties}.
\newblock 2016.

\bibitem[Rob]{Robles}
C.~Robles.
\newblock {Schubert varieties as variations of Hodge structure}.
\newblock {\em Selecta Mathematica}, 20(3):719--768, 2014.

\bibitem[Sat]{Satake}
I.~Satake.
\newblock {Holomorphic imbeddings of symmetric domains into a Siegel space}.
\newblock {\em American Journal of Mathematics}, 87(2):425--461, 1965.

\end{thebibliography}
 
\end{document}